\documentclass[12pt]{article}
\usepackage{amsthm,bm,amssymb}
\usepackage{color}
\usepackage{graphicx}
\usepackage{fancyhdr}
\usepackage{verbatim}
\usepackage{mathrsfs}
\usepackage{amsfonts,amscd,amssymb,epsfig,latexsym,amsmath,enumerate}
\usepackage{mathtools}
\usepackage{verbatimbox}
\usepackage{tikz}
\usepackage{tikz-cd}
\usepackage{bbding}
\usepackage{dsfont}
\usepackage{slashed}
\usepackage{hyperref}
\hypersetup{
    colorlinks,
    citecolor=black,
    filecolor=black,
    linkcolor=black,
    urlcolor=black
}

\usepackage{yfonts}

\usepackage{graphicx}
\usepackage[utf8]{inputenc}
\usepackage[export]{adjustbox}
\usepackage{wrapfig}

\graphicspath{ {images/} }
\usepackage{float}

\usetikzlibrary{matrix,arrows,decorations.pathmorphing}
\usepackage{enumerate}
\usepackage [english]{babel}
\usepackage [autostyle, english = american]{csquotes}
\MakeOuterQuote{"}

\newtheorem{theorem}{Theorem}[section]
\newtheorem{lemma}[theorem]{Lemma}
\newtheorem{proposition}[theorem]{Proposition}
\newtheorem{corollary}[theorem]{Corollary}

\newtheorem{definition}[theorem]{Definition}
\newtheorem{ex}[theorem]{Example}
\newenvironment{example}{\begin{ex} \rm}{\end{ex}}

\renewcommand{\caption}{}
\numberwithin{equation}{section}
\def\&{\wedge}

\newcommand{\bb}{\mathbb}

\newcommand{\R}{\bb{R}}
\newcommand{\C}{\bb{C}}
\newcommand{\Q}{\bb{Q}}
\newcommand{\Z}{\bb{Z}}
\newcommand{\N}{\bb{N}}

\DeclareMathOperator{\id}{id}

\usepackage{geometry}
\title{Local Lagrangian Floer Homology of Quasi-Minimally Degenerate Intersections}
\author{Shamuel Auyeung}

\begin{document}
	\maketitle
	
	\begin{abstract}
		\noindent We define a broad class of local Lagrangian intersections which we call quasi-minimally degenerate (QMD) before developing techniques for studying their local Floer homology. In some cases, one may think of such intersections as modeled on minimally degenerate functions as defined by Kirwan. One major result of this paper is: if $L_0,L_1$ are two Lagrangian submanifolds whose intersection decomposes into QMD sets, there is a spectral sequence converging to their Floer homology $HF_*(L_0,L_1)$ whose $E^1$ page is obtained from local data given by the QMD pieces. The $E^1$ terms are the singular homologies of submanifolds with boundary that come from perturbations of the QMD sets. We then give some applications of these techniques towards studying affine varieties, reproducing some prior results.
	\end{abstract}
	
	\tableofcontents
	
	\section{Introduction}
	
	The aim of this paper is to present a general framework which aids in the computation of certain Floer-theoretic invariants using local data. The main inspiration is a large class of minimally degenerate functions, a term coined in Kirwan's thesis \cite{Kirwan}. Here, we give a definition that is somewhat more general than the one Kirwan originally gave and later on in Section 2, will further generalize the definition. However, in more recent work with Penington, Kirwan has considered very general functions as well \cite{KirwanPenington}.
	
	\begin{definition} \label{mindeg}
		Let $f : M \to \R$ be a smooth function on a manifold $M$. A set $C$ is called \textbf{minimally degenerate} if the following conditions hold.
		\begin{enumerate}
			\item $C$ is a compact set contained in the set of critical points for $f$ and $f$ is constant on $C$. $C$ has an isolating open neighborhood $V$ which means that inside of $V \setminus C$, $f$ does not have any critical points. Such a $C$ is called a critical subset of $f$.
			
			\item There is a submanifold $S$ containing $C$ such that $f|_S$ takes $C$ as its minimum set.
			
			\item At every point $x \in C$, the tangent space $T_x S$ is maximal among all subspaces of $T_x M$ on which the Hessian $\text{Hess}_x f$ is positive semi-definite (synonymously, nonnegative).
		\end{enumerate}
		
		\noindent If the critical points of $f$ is a disjoint union $\bigsqcup C$ where each $C$ is minimally degenerate, then we say that $f$ is minimally degenerate.
	\end{definition}
	
	In effect, minimal degeneracy means that critical sets can be as degenerate as minima but
	no worse. This large class of functions includes many interesting examples such as Morse-Bott functions or functions on varieties which have subvarieties with singularities as critical sets (see Section 2). In symplectic geometry, the importance of this definition arises when considering the norm square of a moment map $|\mu|^2$. In general, $|\mu|^2$ is not Morse-Bott but may be minimally degenerate and hence, can still be studied via Morse theoretic techniques. For this reason, Definition \ref{mindeg} is sometimes referred to as a Morse-type definition. Kirwan applied such tools to $|\mu|^2$ which are a major element of her proof of Kirwan surjectivity. This result is celebrated for its importance towards studying symplectic quotients and geometric invariant theory.
	
	\subsection{Organization of the Paper}
	
	In Section 2, we first summarize some of the properties of minimal degenerate functions before expanding on Definition \ref{mindeg} by introducing the definitions of \textbf{flattened degeneracy} and \textbf{quasi-minimal degeneracy}. As is often the case in mathematics, defining something isn't difficult but defining something \textit{useful} can be. We hope to demonstrate the usefulness of these definitions by proving a series of results. The main result is the existence of a $C^1$-small perturbation which enlarges a flattened degenerate critical set into a submanifold with boundary without changing the homotopy type of the critical set. 
	
	In Section 3, we define flattened and quasi-minimal degeneracy for a subset $C$ of the intersection of a pair of Lagrangians. Part of the definition involves a submanifold $S$, much like in Definition \ref{mindeg} and the definitions of Section 2. In fact, although the Lagrangian definition of quasi-minimal degeneracy is fairly general, in some cases, one can think of $C$ as locally modeled on minimally degenerate functions. Indeed, later on in Section 3, we prove a result relating the ``Morse'' and Lagrangian definitions of quasi-minimal degeneracy in the case that one of the Lagrangians is the graph of an exact 1-form.
	
	Before proving the ``Morse'' implies Lagrangian result however, we establish the existence of a $C^1$-small perturbation of the Lagrangians locally around an isolated subset $C$ of the intersection. Like the ``Morse'' case, the perturbation yields a codim 0 submanifold with boundary $\Sigma$ of $S$. The process by which we do this can be intuitively thought of as ``thickening'' the intersection and we shall refer to these $\Sigma$ as \textbf{thickenings} of $C$. This is the content of Theorem \ref{perturb} and is the key technical result which we use to extend a result of Pozniak in Section 5. One primary motivation behind constructing such a specific perturbation is this: in Floer theory, genericity is a double-edged sword. For example, a small generic perturbation of a Hamiltonian function results in gaining the favorable property of nondegeneracy yet the perturbed function can hardly be studied precisely because it is generic. Therefore, it is often more helpful to perturb in a controlled way at the expense of having some amount degeneracy as a result.
	
	The purpose of Section 4 is to give an exposition of the results from Pozniak's thesis \cite{Pozniak}. The main result that we will use is Theorem \ref{Pmain_thm}. It roughly says: if two Lagrangians intersect cleanly in an isolated neighborhood, then the local Floer homology is determined by the singular homology of the clean intersection. The necessary definitions for understanding this theorem are provided in the section.
	
	As stated above, in Section 5, we use Theorem \ref{perturb} to extend Pozniak's main result to give a stronger Theorem \ref{main_thm}: if two Lagrangians have a quasi-minimally degenerate intersection in an isolated neighborhood, then the local Floer homology is determined by the singular homology of the intersection.
	
	In Section 6, we extract a spectral sequence from Theorem \ref{main_thm}, much in the same way that Seidel extracted a spectral sequence from Theorem \ref{Pmain_thm} in \cite{SeidelKnot}. The local data obtained from the isolated neighborhoods form the $E^1$ page of the spectral sequence and converges to the (global) Lagrangian Floer homology. For the sake of simplifying the exposition, we shall ignore orientations and work with $\Z_2$ coefficients.
	
	\begin{theorem}
		Suppose $L_0 \cap L_1$ decomposes into $\bigsqcup C_p$ where each $C_p$ is quasi-minimally degenerate. Let $\Sigma_p := \Sigma_{C_p}$ be the thickening for $C_p$. Then there is a spectral sequence which converges to $HF_*(L_0,L_1)$ and whose $E^1$-term is
		\[E^1_{pq} =
		\begin{cases}
			H_{p+q-\iota(\Sigma_p)}(C_p;\Z/2), & 1 \leq p \leq r; \\
			0, & \text{otherwise}.
		\end{cases}\]
	\end{theorem}
	
	As applications, we perform four brief demonstrations in Section 7: we compute the Hamiltonian Floer homology of a particular affine variety, give an alternative method for studying certain manifolds with corners, study the $E^1$ page for a particular log Calabi-Yau, and show how the spectral sequence may be applied to situations beyond that of log Calabi-Yau.
	
	The first and fourth examples may have been computed before but the author does not know where they may appear in the literature. The second and third examples have been previously computed but relied heavily on structure which would not be available in more general settings. For example, Ganatra and Pomerleano in \cite{GanatraPomerleano} computed local Hamiltonian Floer cohomology of certain types of minimally degenerate families of orbits which appear as manifolds with corners. In this paper, our result is able to compute local Hamiltonian Floer cohomology for all such families of minimally degenerate orbits. In a different vein, Pascaleff, in \cite{Pascaleff1}, computed wrapped Lagrangian Floer cohomology of certain Lagrangian sections in a log Calabi-Yau surface. In this paper, we indicate how the spectral sequence aids in computing wrapped Floer cohomology for many other Lagrangians inside smooth affine surfaces beyond the log Calabi-Yau case.
	
	The final section is less mathematical and more conjectural. We speculate about other applications and research directions of minimal degeneracy.
	
	\subsection{Acknowledgements} 
	I would like to thank my advisor Mark McLean for his invaluable encouragement and guidance. Indeed, this project was first sketched in a talk he gave \cite{McLean3} and I am grateful for the opportunity to work on the details and also discover some further directions. I am also very grateful to Andrew Hanlon, Jiahao Hu, Lisa Marquand, Kevin Sackel, and Yao Xiao for fruitful discussions and to Aleksandar Milivojevic for introducing me to mathcha.io, the tool used to make most of the figures in this paper.
	
	\section{Definitions and Basic Results}
	
	In the introduction, Definition \ref{mindeg} tells us what it takes for a set to be minimally degenerate and also what it means for functions to be minimally degenerate. Here are two concrete examples to compare. One Morse-Bott example is that of the height function $h$ on a torus, ``laid on its side.'' The critical submanifold for the height function is a disjoint union of two circles; the function takes its maximum on one circle and minimum on the other. Call $C_M$ the circle on which it is a maximum. The definition of a minimally degenerate function requires that $C_M$ is contained in a submanifold $S$ such that when restricting $h$ to $S$, $h|_S$  takes its minimum on $C_M$. Here, it is convenient to simply let $S:=C_M$ so that $h|_S$ is constant and thus, $C_M$ is both the maximum and minimum set of $h|_S$. Here is a cartoon of the situation where we depict only one of the circles.
	
	\begin{center}
		\tikzset{every picture/.style={line width=0.75pt}} 
		
		\begin{tikzpicture}[x=0.75pt,y=0.75pt,yscale=-1,xscale=1]
			
			\draw    (101,93) .. controls (115,84) and (136,80) .. (146,96) ;
			\draw    (95,91) .. controls (113,111) and (134,110) .. (153,93) ;
			\draw  [dash pattern={on 0.84pt off 2.51pt}]  (83,98) .. controls (93,63) and (166,73.5) .. (165,95) ;
			\draw  [dash pattern={on 0.84pt off 2.51pt}]  (83,98) .. controls (80,117) and (164,120.5) .. (165,95) ;
			\draw    (251,62.5) -- (251,126.5) ;
			\draw    (196,95.5) -- (235,95.5) ;
			\draw [shift={(237,95.5)}, rotate = 180] [color={rgb, 255:red, 0; green, 0; blue, 0 }  ][line width=0.75]    (10.93,-3.29) .. controls (6.95,-1.4) and (3.31,-0.3) .. (0,0) .. controls (3.31,0.3) and (6.95,1.4) .. (10.93,3.29)   ;
			\draw   (67,95) .. controls (67,77.6) and (92.07,63.5) .. (123,63.5) .. controls (153.93,63.5) and (179,77.6) .. (179,95) .. controls (179,112.4) and (153.93,126.5) .. (123,126.5) .. controls (92.07,126.5) and (67,112.4) .. (67,95) -- cycle ;
			
			\draw (209,74.4) node [anchor=north west][inner sep=0.75pt]    {$h$};
			\draw (263,84.4) node [anchor=north west][inner sep=0.75pt]    {$\mathbb{R}$};
		\end{tikzpicture}
		
		\vspace{5mm}
		\caption{The height function on a torus}
	\end{center}
	
\begin{example} \label{qmdex}
	A minimally degenerate example to have in mind begins with a compact genus-2 surface $M$ embedded in $\R^3$ (we'll suppress notation ordinarily used to denote embeddings). The embedding is such the height function $h(x,y,z) = z$ has critical points which form two figure 8's—a subvariety, call them $E_1$ and $E_2$. Observe that the dimension of $\ker \text{Hess} \, h$ is not constant along the connected components of the critical points. For $E_1$, the minimum figure 8, $M$ itself serves as the needed submanifold containing this minimal set. Here is a picture.
	
	\begin{center}
		\tikzset{every picture/.style={line width=0.75pt}} 
		
		\begin{tikzpicture}[x=0.75pt,y=0.75pt,yscale=-1,xscale=1]
			
			\draw    (101,93) .. controls (115,84) and (136,80) .. (146,96) ;
			\draw    (194,92) .. controls (208,83) and (229,79) .. (239,95) ;
			\draw    (95,92) .. controls (113,112) and (134,111) .. (153,94) ;
			\draw    (187,92) .. controls (205,112) and (226,111) .. (245,94) ;
			\draw  [dash pattern={on 0.84pt off 2.51pt}]  (83,98) .. controls (93,63) and (134,62) .. (165,95) ;
			\draw  [dash pattern={on 0.84pt off 2.51pt}]  (165,95) .. controls (191,129) and (254,127) .. (256,102) ;
			\draw  [dash pattern={on 0.84pt off 2.51pt}]  (165,103) .. controls (185,67) and (257,66) .. (256,102) ;
			\draw  [dash pattern={on 0.84pt off 2.51pt}]  (83,98) .. controls (80,117) and (136,130) .. (165,103) ;
			\draw    (341,47.5) -- (341,143.5) ;
			\draw    (290,95.5) -- (329,95.5) ;
			\draw [shift={(331,95.5)}, rotate = 180] [color={rgb, 255:red, 0; green, 0; blue, 0 }  ][line width=0.75]    (10.93,-3.29) .. controls (6.95,-1.4) and (3.31,-0.3) .. (0,0) .. controls (3.31,0.3) and (6.95,1.4) .. (10.93,3.29)   ;
			\draw    (81,67.5) .. controls (121,37.5) and (148,94.5) .. (181,63.5) ;
			\draw    (181,63.5) .. controls (211,32.5) and (260,71.5) .. (264,86.5) ;
			\draw    (147,132) .. controls (187,102) and (207,162) .. (247,132) ;
			\draw    (73,111.5) .. controls (77,130.5) and (119,152.5) .. (147,132) ;
			\draw    (73,111.5) .. controls (69,100.5) and (67,80.5) .. (81,67.5) ;
			\draw    (247,132) .. controls (263,122) and (275,109.5) .. (264,86.5) ;
			
			\draw (304,72.4) node [anchor=north west][inner sep=0.75pt]    {$h$};
			\draw (347,86.4) node [anchor=north west][inner sep=0.75pt]    {$\mathbb{R}$};
		\end{tikzpicture}
		
		\caption{The height function on a genus 2 surface}
	\end{center}
	
	\noindent However, the maximum figure 8 $E_2$ does not have a submanifold $S$ containing it such that $h|_S$ takes a minimum on this figure 8. So as it stands, $h$ is not minimally degenerate though $E_1$ is minimally degenerate. If we perturb $h$ locally around $E_2$ so that its new maximum is achieved only at a single point $p$, then the new function is minimally degenerate.
\end{example}

\subsection{Comparing Kirwan's Original Definition to Definition \ref{mindeg}}
	
	Having seen some examples, it is worth pointing out that though Definition \ref{mindeg} is similar to one found in Kirwan's thesis \cite{Kirwan} (p. 65), there are a few important differences.
	
	Firstly, we focus on individual critical sets $C$ because we wish to later consider isolated sets $C \subset \Lambda \cap L$ that are contained in the intersection of Lagrangian submanifolds. Such a $C$ has no intrinsic reference to a smooth function but nonetheless, may display minimal degeneracy type properties such as admitting a submanifold $S$ with some nice properties. This will be made precise later. 
	
	Moreover, we make no assumptions about the normal bundle of $S$ and the relevant restrictions of the Hessian are positive semi-definite instead of positive definite. The first relaxing of the definition is simply because we don't need the assumptions but the second condition is quite crucial and will be explained in due time. There is also a third difference in definition: we don't require the critical set to be a finite union but instead, we require the compact sets to have isolating neighborhoods.
	
	This third difference is made for two reasons. The first is that we wish to avoid certain pathological compact sets such as $A = \{1/n: n \in \N\} \cup \{0\} \subset \R$ or the ``Hawaiian earring'' 
	\[\mathcal{H}=\bigcup_{n=1}^\infty \left\{(x,y) \in \R^2 \mid \left(x-\frac {1}{n} \right)^2 + y^2 = \left(\frac {1}{n}\right)^2\right\}.\]
	
	\noindent Indeed, if $A$ were to arise as the minimum set of some smooth function, then somewhere between each $\frac{1}{n}$ and $\frac{1}{n+1}$ would be a maximum. These maxima would converge towards 0 and hence, $A$ is not isolated. A similar argument also shows that $H$ is not isolated. In point of fact, Kirwan's definition also prohibits such closed sets.
	
	However, finite unions are not general enough in Floer theory; one often encounters infinite unions of Reeb orbits. So in order to continue to prohibit pathological closed sets but also expand the definition to allow for infinite unions, we've opted to use isolated closed sets in our definition.
	
	\subsection{Generalization of Minimal Degeneracy}
	
	Since Lagrangian intersections are our main motivation, consider the following example. 
	
	\begin{example}
		
		Let $L_0$ be the zero section of $T^* \R \cong \R^2$ with standard symplectic form and $L_1$ be the graph of $df= 2x$. The linear symplectomorphisms on $\R^2$ can be thought of as elements of $SL(2,\R) \cong Sp(2,\R)$. One such example is the shearing map represented by 
		\[\begin{pmatrix}
			1 & -1 \\ 0 & 1
		\end{pmatrix},\]
		
		\noindent which sends $L_1$ to the graph of $df=-2x$ while fixing $L_0$. Hence, before applying the linear symplectomorphism, $L_1$ is described by the Morse function $f(x) = x^2$ and afterwards, described by $-f$.
	\end{example}
	
	This example illustrates that even in the case of transverse intersections, a choice of Weinstein neighborhood affects whether the intersection behaves like a minimum or a maximum. Indeed, one may construct examples of Lagrangians $L_0$ and $L_1$ intersecting transversally at a point $p$ and then choose a Weinstein neighborhood so that $L_1$ is the graph of $df$ where $f$ is a Morse function with a critical point at $p$ of arbitrary index. Hence, in the Lagrangian setting, any attempt to define $C \subset L_0 \cap L_1$ to be minimally degenerate should not rely on a Weinstein neighborhood since, depending on the neighborhood, $C$ may or may not be minimal. Thus, this motivates us to give a few definitions that generalize Definition \ref{mindeg}.
	
	\begin{definition}
		Let $f: M \to \R$ be a smooth function and let $C$ be an isolated family of critical points of $f$. Let $S \subset M$ be a submanifold containing $C$. We say that $f$ is \textbf{flattened degenerate along} $(C,S)$ if:
		\begin{enumerate}
			\item $f|_S$ is minimal along $C$.
			\item $\ker \text{Hess}_x f = T_x S$ for all $x \in C$.
		\end{enumerate}
		If the critical points of $f$ form a disjoint union $\bigsqcup C$ where each $C$ has a submanifold $S_C$ such that $f$ is flattened degenerate along $(C,S_C)$, we will say that $f$ is flattened degenerate.
	\end{definition}

Observe that if we take a smaller submanifold $S' \subset S$ that still contains $C$, then $f$ is also flattened degenerate along $(C,S')$. Now, for an example:
	
\begin{example} \label{flatex}
	Consider the genus 2 surface $M$ from Example \ref{qmdex}, embedded into $\R^3$. We may deform the embedding by ``flattening'' the bottom of the surface so that the height function has a set of minima $C$ that looks like a ``mask'' (see figure below) and is a codim 0 submanifold-with-boundary. If we pick local coordinates, along $C$, $h$ is constant and hence its 2nd derivatives and hence, Hessian, is trivial along $C$. Thus, the height function is flattened degenerate along $(C,M)$. It's obvious here but worth pointing out that the submanifold $S$ we chose is $M$ itself).
	
	\begin{center}
\tikzset{every picture/.style={line width=0.75pt}} 

\begin{tikzpicture}[x=0.75pt,y=0.75pt,yscale=-1,xscale=1]

\draw  [draw opacity=0] (111.26,120.2) .. controls (108.06,121.36) and (104.6,122) .. (101,122) .. controls (84.43,122) and (71,108.57) .. (71,92) .. controls (71,75.43) and (84.43,62) .. (101,62) .. controls (104.6,62) and (108.06,62.64) .. (111.26,63.8) -- (101,92) -- cycle ; \draw   (111.26,120.2) .. controls (108.06,121.36) and (104.6,122) .. (101,122) .. controls (84.43,122) and (71,108.57) .. (71,92) .. controls (71,75.43) and (84.43,62) .. (101,62) .. controls (104.6,62) and (108.06,62.64) .. (111.26,63.8) ;  
\draw  [draw opacity=0] (149.74,63.8) .. controls (152.94,62.64) and (156.4,62) .. (160,62) .. controls (176.57,62) and (190,75.43) .. (190,92) .. controls (190,108.57) and (176.57,122) .. (160,122) .. controls (156.4,122) and (152.94,121.36) .. (149.74,120.2) -- (160,92) -- cycle ; \draw   (149.74,63.8) .. controls (152.94,62.64) and (156.4,62) .. (160,62) .. controls (176.57,62) and (190,75.43) .. (190,92) .. controls (190,108.57) and (176.57,122) .. (160,122) .. controls (156.4,122) and (152.94,121.36) .. (149.74,120.2) ;  
\draw    (111.26,63.8) .. controls (123,71.5) and (139,70.5) .. (149.74,63.8) ;
\draw    (111.26,120.2) .. controls (124,114.5) and (138,115.5) .. (149.74,120.2) ;
\draw   (87.25,92) .. controls (87.25,84.41) and (93.41,78.25) .. (101,78.25) .. controls (108.59,78.25) and (114.75,84.41) .. (114.75,92) .. controls (114.75,99.59) and (108.59,105.75) .. (101,105.75) .. controls (93.41,105.75) and (87.25,99.59) .. (87.25,92) -- cycle ;
\draw   (146.25,92) .. controls (146.25,84.41) and (152.41,78.25) .. (160,78.25) .. controls (167.59,78.25) and (173.75,84.41) .. (173.75,92) .. controls (173.75,99.59) and (167.59,105.75) .. (160,105.75) .. controls (152.41,105.75) and (146.25,99.59) .. (146.25,92)  -- cycle ;
\end{tikzpicture}
\end{center}
	
\end{example}

As this example illustrates, one way to obtain flattened degenerate functions is to perform this flattening process. The following lemma demonstrates the usefulness of this definition and also that the flattening procedure can always be applied to flattened degenerate functions along $(C,S)$ so that we get a new submanifold-with-boundary $\Sigma$. Moreover, $C$ will be homotopy equivalent to $\Sigma$. In the example above, there is no need to undergo this procedure since $C$ itself is already a submanifold-with-boundary.
	
	\begin{lemma} \label{fperturb}
		If $f$ is flattened degenerate along $(C,S)$, then there is a codim 0 submanifold $\Sigma \subset S$ with boundary which is an isolated critical set of $\check{f}$ containing $C$, a function that is $C^1$ close to $f$. Moreover, $C \hookrightarrow\Sigma$ is a homotopy equivalence. 
	\end{lemma}
	
	\begin{proof}
		\textbf{Case 1:} Suppose $\dim S = \dim M$. In this case, $f$ is minimal on $C$ and $\text{Hess}_x f = 0$ for all $x \in C$. Let $\delta > 0$ (a parameter we may adjust) and $\rho:\R \to \R$ be a smooth function such that $\rho(x) = 0$ for $x \leq \delta/2$, $\rho(x) = x$ for $x \geq \delta$, and $0 < \rho'(x) < 3$ for $x \in (\delta/2, \delta)$. 
		
		\begin{center}
			\tikzset{every picture/.style={line width=0.75pt}} 
			\begin{tikzpicture}[x=0.75pt,y=0.75pt,yscale=-1,xscale=1]
				
				\draw    (160,50) -- (160,240) ;
				\draw    (160,240) -- (350,240) ;
				\draw [line width=1.5]    (260,140) -- (350,50) ;
				\draw [line width=1.5]    (210,240) .. controls (247,239.83) and (242.33,157.5) .. (260,140) ;
				\draw [line width=1.5]    (160,240) -- (210,240) ;
				\draw    (260,230) -- (260,250) ;
				\draw    (210,230) -- (210,250) ;
				
				\draw (314,92.4) node [anchor=north west][inner sep=0.75pt]    {$\id$};
				\draw (203,252.4) node [anchor=north west][inner sep=0.75pt]    {$\frac{\delta}{2}$};
				\draw (231,162.4) node [anchor=north west][inner sep=0.75pt]    {$\rho$};
				\draw (255,252.4) node [anchor=north west][inner sep=0.75pt]    {$\delta$};
				
			\end{tikzpicture}
		\end{center}	
		
		Then, let $\check{f} = \rho \circ f$. Note that $\check{f}^{-1}(0) = f^{-1}([0,\delta/2])$ and by Sard's theorem, for generic $\delta$, $f^{-1}(\delta)$ is a submanifold. Hence, $\Sigma := \check{f}^{-1}(0)$ is a submanifold with boundary and clearly a critical set of $\check{f}$. Since $C$ is isolated by some open set $U$, we may choose $\delta$ small enough such that $\Sigma \subset U$. We may also choose $\delta$ small enough so that the vector field $-\nabla \check{f}$ has complete flow. This gives the desired homotopy inverse to $C \hookrightarrow \Sigma$ since the only points on $\Sigma$ that are stationary are points of $C$ and in the limit, the flow of the gradient of any point goes to a point in $C$ . The bounds on the 1st derivative plus the fact that $\text{Hess}_x f = 0$ for $x\in C$ makes $\check{f}$ $C^1$-close to $f$. \\
		
		\noindent \textbf{Case 2:} $\dim S < \dim M$. We work in a tubular neighborhood of $S$ which is diffeomorphic to the normal bundle of $S$: $\nu:NS \to S$. Then, of course $S$ is a codim 0 submanifold of $S$ and $f|_S$ has $C$ as a minimum and also satisfies the Hessian condition. We may apply Case 1 to this and obtain a codim 0 submanifold $\Sigma \subset S$ which is the critical set of $\rho \circ f|_S$.
		
		Next, we wish to pullback $\rho \circ f|_S$ to $M$ in some way. There are two possibilities: we may take $\nu^* (\rho \circ f)$ which is constant on the fibers. Hence, the critical set of this function is $NS|_\Sigma$, not merely $\Sigma$. However, $NS|_\Sigma$ is homotopy equivalent to $\Sigma$ due to the contractible fibers. We may then extend this function from the tubular neighborhood to all of $M$ via bump functions.
		
		Alternatively, we fix a complete metric on $NS$ and let $r(x,v) = |v|$ be the radial function with respect to this metric. Then $(1+r^4)(\rho \circ f|_S)$ is a function on $NS$ which has $\Sigma$ as an isolated critical set. We can then extend this function to all of $M$ once again, using bump functions.
	\end{proof}
	
	We think of the process in which $\Sigma$ is obtained as a sort of ``flattening'' process because $\rho$ is constant on $[0,\delta/2]$. We will often refer to a $\Sigma$ obtained in this way as a \textbf{thickening} of $C$. We now give another definition.
	
	\begin{definition} \label{QMD}
		We say that $f$ is \textbf{quasi-minimally degenerate} (QMD) along $C$ if there exists a smooth function $\tau \geq 0$ and a submanifold $S$ so that:
		\begin{enumerate}
			\item $\tau^{-1}(0) = C$.
			\item $\ker \text{Hess}_x\tau$ is transverse to $T_x S$ for each $x \in C$.
			\item $f - \tau$ is flattened degenerate along $(C,S)$.
		\end{enumerate}
		If the critical points of $f$ form a disjoint union $\bigsqcup C$ such that $f$ is QMD along each $C$, then we say that $f$ is quasi-minimally degenerate.
	\end{definition}

	\noindent \textbf{Remark:} Here are a few immediate and important observations.
	\begin{itemize}
		\item The height function in Example \ref{qmdex} is QMD along the set $C$ which is a minimal figure 8. This is because we can find a $\tau \geq 0$ with $\tau^{-1}(0)$ (this is always doable for any closed set $C$ using partition of unity). Then, using $S=M$, the transversality condition is automatically satisfied. Lastly, we can arrange $\tau$ to have these properties and in addition, be such that $f-\tau$ is flattened degenerate. Intuitively, the ``flattening'' procedure of Example \ref{flatex} uses such a $\tau$.
	
		\item Since $\tau \geq 0$ and $C$ is a set of minima for $\tau$, then $(d\tau)|_C = 0$.
		
		\item Note that if we take a smaller submanifold $S'$ inside of $S$ which still contains $C$, there's no issue since all the properties of $\tau$ still hold on $S'$. In this way, we may elect to ``shrink'' $S$ while maintaining the relevant properties. Put another way, $S$ should not be viewed as part of the data but rather the germ of submanifolds containing $C$ is what's essential. Here is an immediate consequence. Note that since $C$ is a minimum for $\tau$, then for each $x \in C$, $\text{Hess}_x \tau \geq 0$ on $T_x S$ and the same holds for points near $C$. Hence, by choosing a small enough $S$, we can assume that for any $x \in S$, $\text{Hess}_x \tau \geq 0$ on $T_x S$ rather than only those $x \in C$.
		
		\item Similarly, for small enough $S$, points $x \in S$ will also be such that $\ker \text{Hess}_x \tau$ is transverse to $T_x S$. The advantage to defining quasi-minimal degeneracy in this way is that we don't need to invoke a metric since $d\tau$ vanishes on $C$. Off of $C$, the Hessian requires a choice of metric.
		
		\item Note that $\frac{d}{ds} (d(f-s\tau)_x) = d\tau_x$ for each $x$. If $x \in C$, then $d\tau_x = 0$ and hence, for $x \in C$, $\frac{d}{ds} (d(f-s\tau)_x) = 0$; i.e. $d(f-s\tau)_x$ is independent of $s$. In particular, set $s=0$ and thus, $d(f-s\tau)_x = df_x$. If $x \notin C$ but is close to $C$, then $d\tau_x \neq 0$ because $\tau^{-1}(0)=C$. So $\frac{d}{ds} (d(f-s\tau)_x) \neq 0$. By being very near $C$, we can assume this means that $d(f-s\tau)_x \neq 0$ as well for any $s \in [0,1]$. Hence, we have a isotopy between a QMD function $f$ and a flattened degenerate function $f-\tau$ where during the isotopy, no new critical points are introduced near $C$ nor are any critical points lost.
	\end{itemize}
	
	A less immediate observation is that, by adding one more condition, we have a function that is minimally degenerate in the spirit of Kirwan's definition. More precisely:	
	
	\begin{lemma}
		If, in addition to the properties listed in \ref{QMD}, $\text{Hess}_x(f-\tau)$ has no positive eigenvalues, then $C$ is minimally degenerate in the sense of Definition \ref{mindeg}. Conversely, if $C$ is a minimally degenerate set, then $f$ is QMD along $C$ and $\text{Hess}_x (f-\tau)$ has no positive eigenvalues.
	\end{lemma}
	
	\begin{proof}
		$(\Longrightarrow)$ $f-\tau$ is flattened degenerate which means $(f-\tau)|_S$ has $C$ as minimum and $\ker \text{Hess}_x(f-\tau) = T_x S$ for $x \in C$. This means that along $C$, the 1st order derivatives of $f-\tau$ restricted to the directions tangent to $S$ are not varying. In other words, $f-\tau$ has $S$ as a critical set: $d(f-\tau)_x = 0$ for $x \in S$ (if necessary, we shrink $S$, treating it as a germ). Hence, $df_x = d\tau_x$ for $x \in S$. Since $\tau$ has $C$ as minimum, $f|_S$ has $C$ as minimum.	
		
		Let $x \in C$ and consider $\text{Hess}_x\, f$. Let $V \subset \ker \text{Hess}_x \,\tau$ be a subspace transverse to $T_x S$ satisfying $V \cap T_x S = 0$ (i.e. it is of complementary dimension). Now, $\text{Hess}_x(f-\tau) = \text{Hess}_x \,f - \text{Hess}_x \, \tau$ has no positive eigenvalues (the additional property mentioned above). And restricting $\text{Hess}_x \, \tau$ to $V$ (which is in its own kernel) gives a trivial quadratic form. Hence, $\text{Hess}_x \,f$ restricted to $V$ is negative definite. 
		
		Also, since $\ker \text{Hess}_x(f - \tau) = T_x S$, we have that $\text{Hess}_x \,f$ and $\text{Hess}_x \, \tau$ agree when restricted to $T_x S$. $\tau \geq 0$ so $\text{Hess}_x \, \tau$ is non-negative definite on $T_x S$ and hence, so is $\text{Hess}_x \, f$. Together, these two facts show that $T_x S$ is the maximal subspace for which $\text{Hess}_x\,f$ is non-negative definite. \\
		
		\noindent $(\Longleftarrow)$ Conversely suppose that $C$ is a connected, isolated minimally degenerate critical locus of $f$ in the sense of Kirwan's thesis (generalized slightly in this paper) and let $S$ be the corresponding ``minimizing submanifold.'' We construct the auxiliary function $\tau$ as follows: By using a complete metric on $M$, we can identify a neighborhood of $S$ with a tubular neighborhood $U \subset NS$ of its normal bundle $NS$. Let $r : U \to [0,\infty)$ be the radial coordinate for this tubular neighborhood; i.e. $(x,v) \in U$ is mapped to the norm $|v| \in [0,\infty)$.  Let $\pi : U \to S$ be the projection map. We define $\tau = r^4 + \pi^*(f|_S)$ on $U$ and then use a bump function to extend $\tau$ to the whole of $M$.
		
		Note that $\tau^{-1}(0) \subset S$ because if we have $(x,v) \in U$ where $v \neq 0$, then $|v|^4 > 0$. Since $f|_S$ has $C$ as its minimum (we can assume it takes values 0), we conclude that $C = \tau^{-1}(0)$. This tells us that $d\tau|_C = 0$, thanks to $C$ being a set of critical points of $f$. Also, when restricted to $S$, the $r^4$ part vanishes and $\pi$ is trivial. So then, $d(f-\tau)|_S = 0$ because $f-\tau$ vanishes along $S$. Moreover, $S$ is the minimum for $f-\tau$ and so $f-\tau$ is negative definite along $NS$. This means that $\ker \text{Hess}(f-\tau) = T_x S$.
		
		Lastly, when restricted to $S$, $r^4$ vanishes. So for a point $x \in S$, we only need to consider the $\pi^*(f|_S)$ piece of $\tau$. But along a fiber of $U$, this is constant and hence the Hessian of $\pi^*(f|_S)$ vanishes on $NS$. This means that $NS \subset \ker \text{Hess}\, \tau$ which implies that this kernel is transverse to $T_x S$.
	\end{proof}
	
	This lemma shows us that a function being minimally degenerate is equivalent to it having a non-negative function $\tau$ with some properties, the essential one being that $f-\tau$ is flattened degenerate. We will shortly see the usefulness of this notion when studying Lagrangian intersections.
	
	\section{Lagrangian Quasi-Minimal Degeneracy}
	
	Similar to above, we will give two definitions concerning the intersection of any pair of Lagrangians.
	
	\begin{definition}
		Let $\Lambda, L \subset M$ be two Lagrangian submanifolds of a symplectic manifold $(M,\omega)$ and $C \subset \Lambda \cap L$. We say that $C$ is \textbf{flattened degenerate along a submanifold} $S \subset \Lambda$ with respect to $\Lambda, L$ if:
		\begin{enumerate}
			\item $T_x S = T_x \Lambda \cap T_x L$ for each $x \in C$.
			\item There exists a time dependent Hamiltonian $H_t$ whose derivative and Hessian vanish along $C$ and satisfies $\frac{d}{dt}(H_t) \geq 0$.
			\item $\phi^H_1(S) \subset L$.
		\end{enumerate}
	\end{definition}
	
	\noindent \textbf{Remark:} Here are some important points. 
	\begin{itemize}
		\item As promised, this definition is intrinsic in the sense that it does not depend on a choice of Weinstein neighborhood. 
		
		\item The time-dependence of $H$ is natural in symplectic geometry and gives a more flexible definition than requiring an autonomous Hamiltonian. 
		
		\item Much in the case of functions, the submanifold $S$ is best thought of as a germ since we may ``shrink'' the submanifold to some $S'$ and use the same $H$ without modification since it has all the same properties on the subset $S' \subset S$. Therefore, some of the remarks we made for functions also applies here.
	\end{itemize}
	
	\begin{theorem} \label{perturb}
		Suppose that $C$ is flattened degenerate along $S$ with respect to $\Lambda,L$. There exists a $C^1$-close family of Lagrangians $Q_s$ with $Q_0 =L$ realized by a Hamiltonian isotopy such that:
		\begin{enumerate}
			\item There exists a fixed open neighborhood $V$ of $C$ such that $Q_s \cap \Lambda \cap V$ is a compact,connected subset inside $V$ for every $s \in [0,1]$; i.e. there is a fixed isolating neighborhood.
			
			\item The intersection $Q_s \cap \Lambda$ near $C$ is homotopy equivalent to a codim 0 submanifold-with-boundary $\Sigma \subset S$ which contains $C$. Moreover, the inclusion $C \hookrightarrow \Sigma$ is a homotopy equivalence.
		\end{enumerate}
	\end{theorem}
	
	\begin{proof}
		We will break this into two cases: $\dim S = n$ where $\dim M = 2n$ and $\dim S < n$. In reality, the may be treated as one case, as we will indicate. But hopefully, this presentation is more digestible. 
		
		We begin by choosing a Weinstein neighborhood $\mathcal{U}$ of $C \subset \Lambda$. Hence, we may view $\mathcal{U}$ as being symplectomorphic to a neighborhood of $C \subset \Lambda$ inside of $T^*\Lambda$. Let $\pi:T^*\Lambda \to \Lambda$ be the bundle map. Even if we did not have the condition above that $S \subset \Lambda$, viewing $S$ as a germ of a submanifold, we may pick a codim 0 submanifold $S' \subset S$ and project it, via $\pi$, to land in $\Lambda$. Our choice of $S'$ is made so that $\pi(S')$ is a submanifold (without boundary) of $\Lambda$. This shows that we do not really lose any generality by assuming $S \subset \Lambda$. \\
		
		\noindent \textbf{Case 1:} When $\dim S = n$, then for $x \in C$, $T_x S = T_x \Lambda \cap T_x L$ implies that $T_x \Lambda = T_x L$; i.e. $\Lambda$ and $L$ are tangent along $C$. This means that along $C$, $L$ is transverse to the fibers of $T^* \Lambda$ and thus, is a graph of some section. Since $L$ is Lagrangian, the section is a closed 1-form and hence, a locally exact 1-form. This means that near $C$, $L$ is the graph of some $df_1$. In fact, each Lagrangian in the family $\Lambda_t = \phi^H_t(\Lambda)$ is a graph of some $df_t$. Now, condition (2) of the definition of flattened degenerate tells us that $dH_t|_C = 0$ and for each $x \in C$, $\text{Hess}_x H_t = 0$, and $\frac{d}{dt}H_t \geq 0$. This implies that $C \subset \Lambda_t \cap \Lambda$ for each $t$.
		
		\begin{lemma}
			Let $\Lambda_t = \phi^H_t(\Lambda)$ be a family of Lagrangians in $T^* \Lambda$ where $H$ has the properties above. Then, $\Lambda_t$ is the graph of some $df_t$. There exists a time-dependent vector field $\widehat{Z}_t$ on $T^* \Lambda$ such that $\widehat{Z}_t$ vanishes on $C$. Additionally, if $\psi_t$ is the flow of $-\widehat{Z}_t$, then $\frac{d}{dt} f_t \circ \psi_t \geq 0$. In particular, $f_t \geq 0$.
		\end{lemma}
		
		\begin{proof}
			Since $\Lambda_t$ is a graph, at each point of $\Lambda_t$, the tangent space of $T^* \Lambda$ splits into a vertical direction and a ``horizontal'' direction (tangent direction to $\Lambda_t$). We may then split $X_{H_t}$ into two components as well. More precisely, first restrict $X_{H_t}$ to $TT^* \Lambda|_{\Lambda_t}$. Then, it equals $Y_t+Z_t$ where $Y_t$ is tangent to the fibers and $Z_t$ is tangent to $\Lambda_t$.
			
			This horizontal component $Z_t$ may behave in such a way that the functions $f_t$ used to define $\Lambda_t$ are decreasing. However, we may project $Z_t$ to $\Lambda$ where it generates a diffeomorphism on $\Lambda$. If we pull back this projection, we get a vector field on the cotangent bundle (call it $\widehat{Z}_t$) which generates a symplectomorphism. Observe that $Z_t - \widehat{Z}_t$ is a vector field in the fiber direction and also that $Z_t$ and $\widehat{Z}_t$ both vanish on $C$. Let $\psi_t$ be the flow generated by $-\widehat{Z}_t$. The flow counteracts the horizontal movement that comes from the original $Z_t$ and hence, the piece of $X_{H_t}$ that matters when flowing $\psi_t(\Lambda_t)$ is the vertical $Y_t$.
			
			Because $\frac{d}{dt}H_t \geq 0$, the $Y_t$ will only flow the $\psi_t(\Lambda_t)$ in such a way that the defining functions $f_t$ increase as well. That is, $\frac{d}{dt}f_t \circ \psi_t \geq 0$. 
			
			Since $L_0 = \Lambda$, we may assume $f_0 \equiv 0$. Pick $x \in T^* \Lambda$ and $t_0 \in [0,1]$. Then there is a path on the interval $[0,t_0]$ given by $\gamma(t)= \psi_t(\psi^{-1}_{t_0}(x))$ which starts at $\psi^{-1}_{t_0}(x)$ and ends at $x$. Along this path, $\frac{d}{dt} f_t(\gamma(t)) \geq 0$ and $f_0(\gamma(t)) = 0$. Hence, $f_{t_0}(x) \geq 0$.
		\end{proof}

		Let $\delta > 0$ be a parameter and $\rho:\R \to \R$ be the function we saw earlier. Let $\rho_s(x) = (1-s)x + s \rho(x)$ be a linear interpolation of $\rho$. Observe that for $x \in \rho^{-1}(0)$, unless $s=1$, $\rho_s(x) > 0$.
		
		Let $\check{f}_s = \rho_s \circ f_1$. We adjust $\delta$ such that $\check{f}^{-1}_1(0) = (\rho \circ f_1)^{-1}(0) = f_1^{-1}([0,\delta/2])$ is a submanifold with boundary also being a smooth manifold. This is possible since regular values are dense by Sard's theorem. Moreover, we choose $\delta$ to be small enough that so that $\check{f}^{-1}_1(0)$ is contained within a neighborhood $U$ of $C$ in which $f_1$ has no critical points other than those in $C$.
		
		We can say something similar about the $\check{f}_s$. We have that $d\check{f}_s = d\rho_s \circ df_1$. Since $\rho_s(x)=(1-s)x +s \rho(x)$, then $d\rho_s = (1-s) \id + sd\rho$. For $x \in [0,\delta/2]$, $d\rho_s$ acts by scalar multiplication via $1-s$. For $x \in [\delta,\infty)$, $d\rho_s = \id$. Hence, in a neighborhood of $C$ and when $s<1$, the critical points of $\check{f}_s$ are precisely $C$. It is only when $s=1$ do we have $\Sigma:=f_1^{-1}([0,\delta/2])$ as a set of critical points.
		
		Setting $Q_s$ to be the graph of $d\check{f}_s$, we have a family of Lagrangians where $Q_0 = L$ and another description of $\Sigma$ as the intersection $Q_1 \cap \Lambda$. Again, $\Sigma$ is a codim 0 submanifold-with-boundary of $S$ containing $C$. If we pick a complete metric, then $-\nabla f_1$ does not vanish anywhere in $U \setminus C$ and this vector field gives a deformation retract of $\Sigma$ onto $C$. We will postpone showing that the family $Q_s$ form a Hamiltonian isotopy until we introduce Lemma \ref{isotopy} below.
		\\
		
		\noindent \textbf{Case 2:} As before, we work in a Weinstein neighborhood and assume that $S \subset \Lambda$. $\pi:T^* \Lambda \to \Lambda$ is the cotangent bundle. When $\dim S = k < n$, $\pi^{-1}(S) = T^* \Lambda|_S$ is a coisotropic submanifold of $T^*\Lambda$. Coisotropic submanifolds admit foliations and in this case, the foliation is actually a fibration over the leaf space which is symplectomorphic to the symplectic reduction $T^* S$.
		
		In more detail, let $\eta \in T^*_x S$. Then a fiber over $(x,\eta) \in T^* S$ is $F=\{\varphi \in T^*_x \Lambda:\varphi|_{T^*_x S} = \eta \}$. If we use a metric to get an orthogonal decomposition: $T^* \Lambda = T^* S \oplus T^*S^\perp$, all such $\varphi$ decompose as $\varphi = \eta + \alpha$. And the set of $\alpha$ form a vector space. So this fibration is a vector bundle.
		
		So let $p:T^*\Lambda|_S \to T^* S$ be the fibration. As in the first case, we have a family of Lagrangians $\Lambda_t= \phi^H_t(\Lambda) \subset T^* \Lambda$. Then let $\widetilde{\Lambda}_t := p(\Lambda_t \cap T^*\Lambda|_S) \subset T^*S$; it is a Lagrangian. See \cite{McDuffSalamon}, p. 221. \\
		
		\noindent \textbf{Claim:} $\widetilde{\Lambda}_t$ is the graph of some $df_t$ where $f_t:S \to \R$ are functions satisfying $f_0 \equiv 0$ and $f_t \geq 0$.
		
		\begin{proof}
			One way to show that $\widetilde{\Lambda}_t$ is a graph near $C$ is to show that for $x \in C$, $T_x S = T_x \widetilde{\Lambda}_t$. Now, we have the ``clean intersection'' condition that $T_x S = T_x \Lambda \cap T_x L$ for $x \in C$. Also, $dH_t|_C = 0$ and $\text{Hess}_x H_t = 0$ for $x \in C$. This means that the flow does not move $C$ at all and that 1st order derivatives of $H_t$ do not change in any directions at $x \in C$ because the Hessians vanish. Hence, for $x \in C$ $T_x S = T_x \Lambda \cap T_x \Lambda_t$ for each $t$.
			
			Armed with that fact, then projecting to the leaf space, we have for $x \in C$, $T_x S = T_x S \cap T_x \widetilde{\Lambda}_t$ which just means $T_x \widetilde{\Lambda}_t =T_x S$. Arguing as in Case 1, this means that near $C$, each $\widetilde{\Lambda}_t$ is the graph of some section. Being Lagrangians, the sections are closed 1-forms which are locally exact. Hence, near $C$, $\widetilde{\Lambda}_t$ is the graph of $df_t$ where $f_t$ is a smooth family of functions $f_t:S \to \R$ and $f_0 \equiv 0$.
			
			As for showing that $f_t \geq 0$, since $\frac{d}{dt}H_t \geq 0$, the restriction of the $H_t$ to $T^*S$ also satisfies the same property. Hence, the result follows from Lemma 3.3.
		\end{proof}
		
		We are now in a position to apply the same sort of argument towards ``flattening'' the intersection by using a family of functions $\rho_s:\R \to \R$ with the same properties as before. Hence, we obtain a family of Lagrangians $\widetilde{Q}_s$ with $\widetilde{Q}_0 = \widetilde{\Lambda}_1$. Since $\phi^H_1(S) \subset L \cap \phi^H_1(\Lambda)$ and $dH_t$ and the Hessians of $H_t$ vanish on $C$, then along $C$, $\phi^H_1(S)$ is tangent to $S$. Hence, near $C$, $p(\phi^H_1(S) \cap T^*\Lambda|_S)$ is a Lagrangian submanifold of dimension $k$, coinciding with $\widetilde{L} = p(L \cap T^*\Lambda|_S)$. Hence, we'll also denote $\widetilde{Q}_0$ by $\widetilde{L}$. We also denote $\widetilde{\Sigma}:=\widetilde{Q}_1 \cap S$ which is a submanifold with boundary that deformation retracts onto $C$.
		
		Now, recall that for each $x \in C$, $T_x \widetilde{\Lambda}_t = T_x S$ due to $C$ being a minimum set for $f_t$. In particular, $T_x \widetilde{L} = T_x \tilde{\Lambda_1} = T_x S$ for $x \in C$. Now, the Hessian of $\rho_s \circ f_1$ at a critical point will have 1st order terms of $\rho_s$ paired with 2nd order terms of $f_1$ and 2nd order terms of $\rho_s$ paired with 1st order terms of $f_1$. Since the 1st order terms of $f_1$ vanish near $C$ and the 1st order terms of $\rho_s$ vanish near $x=0$, we may conclude that near $C$, both the derivative and Hessian of $\rho_s \circ f_1$ vanish and hence, for each $s$, $T_x \widetilde{Q}_s = T_x \widetilde{L}$ for $x \in C$.
		
		What we need to do now is show that these Lagrangians lift to our original setting; i.e. \\
		
		\noindent \textbf{Claim (restatement of what is to be proved):} There exists a $C^1$-close Lagrangian family $Q_s$ realized by a Hamiltonian isotopy $X_{K_s}$ such that $Q_0 = L$, $T_x Q_s = T_x L$ for all $x \in C$, and $p(Q_s \cap T^* \Lambda|_S)$ is the graph of $d(\rho_s \circ f_1)$. Moreover, the Hamiltonian $K_s$ has vanishing derivative and Hessian on $C$.
		
		\begin{proof}
			The family $\widetilde{Q}_s$ is given by a isotopy $i_s:\widetilde{L} \to (T^*S,d\lambda)$. 
			
			Following \cite{Oh}, in more generality, let $i:[0,1] \times L \to (M,\omega)$ be an isotopy and $L_s$ be the image of $\{s\} \times L$ under $i$. The pullback $i^* \omega$ can be written as $i^* \omega = ds \wedge \alpha + \beta$ where $\alpha$ and $\beta$ both vanish when contracted with $\partial_s$, a vector field tangent to the interval $[0,1]_s$. Let $i_s:L \to [0,1] \times L$ be the natural inclusion.
			
			Next, define $\alpha_s:=\iota_{i_* \partial_s} \omega|_{L_s}$ where $i_* \partial_s$ is the pushforward of $\partial_s$. Note that $i^* \alpha_s(V) = \omega(i_* \partial_s,i_*(V)) = \iota_{\partial_s} i^* \omega = \alpha$. Now, suppose that $i^*_s \alpha$ is closed. This implies that $i^*_s i^* d\alpha_s = 0$. But also, $\alpha_s$ is defined on $L_s$ and the map $i \circ i_s:L \to L_s$ is a diffeomorphism. This implies that $d\alpha_s=0$. Hence, the $\omega$-dual $X_{\alpha_s}$ is a symplectic vector field defined on $L_s$.
			
			Next, by definition, $\alpha_s - \iota_{i_* \partial_s} \omega|_{L_s} = 0$. The $\omega$-dual of this vector field is $V_s:=X_{\alpha_s} - i_*\partial_s$. This vector field $V$ is tangent to $L_s$. To show this, suppose that for any Lagrangian $L$ and vector field $V$, $\iota_V \omega|_L = 0$. This means that $V$ is in $TL^\omega$, the $\omega$-orthogonal space. But $L$ is Lagrangian and so $TL = TL^\omega$. Applying this to our situation, $X_{\alpha_s} - i_*\partial_s \in TL_s$ for each $s$. This means that the vector field $i_* \partial_s$ which realizes the isotopy, can always be upgraded to a symplectic vector field $X_{\alpha_s}$ simply by reparametrizing the domain $L_s$. This was achieved using only the assumption that $i^*_s \alpha$ is closed for each $s$.
			
			The next question is, when is $\alpha_s$ exact? In our situation, the vector fields $X_{\alpha_s}$ vanish on the critical set $C$ and hence, $\alpha_s$ vanishes on $C$. Being closed, we may conclude that $\alpha_s = df_s$ If we take a neighborhood $U$ of $C$ that deformation retracts to $C$, then we may pull $\alpha_s$ back by the retraction, which is a homotopy equivalence, thereby extending the $\alpha_s$ to a neighborhood of $C$. Since de Rham cohomology is a homotopy invariant, this means that the extension is also exact. Hence, the $\omega$-dual of the extension is a Hamiltonian vector field which we'll continue to call $X_{\alpha_s}$. Let us summarize the last few paragraphs as a general lemma.
			
			\begin{lemma} \label{isotopy}
				Let $i:[0,1]_s \times L \to (M,\omega)$ be an isotopy of embeddings and $i_s:L \to [0,1] \times L$ the natural map sending $L$ to $\{s\} \times L$. Then we may write $i^* \omega = ds \wedge \alpha + \beta$ where $\iota_{\partial_s} \alpha = \iota_{\partial_s} \beta = 0$. If $i^*_s \alpha$ is closed, then the vector field $i_* \partial_s$ may be modified to a family of symplectic vector fields $X_s$ defined on $L_s$, the image of $i\circ i_s$. In particular, if $L_0$ is Lagrangian, then each $L_s$ is Lagrangian.
				
				Furthermore, if each of the $X_s$ vanish along a set $C \subset M$, then in a neighborhood of $C$, these $X_s$ may be taken to be Hamiltonian vector fields.
			\end{lemma}
			
			Returning to the proof of the claim, by this lemma, there are time-dependent Hamiltonian vector fields which generate the flow to realize the isotopy $i_s$. Denote the corresponding Hamiltonian by $\widetilde{K}_s$. Because of the tangency of the $\widetilde{Q}_s$ all along $C$, the 1st order derivatives of the Hamiltonian vector fields has to vanish; i.e. $\text{Hess}_x \, \widetilde{K}_s = 0$ for $x \in C$.
			
			The $\widetilde{Q}_s$ differ from each other only in the region of $\kappa := f^{-1}_1([0,\delta])$ because $\rho_s = \id$ on $[\delta,\infty)$. Hence, $\widetilde{K}_s$ is such that the Hamiltonian vector fields vanish on $\widetilde{Q}_s \setminus \kappa$. Let us take $\widetilde{K}_s$ to be constant ouside of this region.
			
			Next, we want to pullback the $\widetilde{Q}_s$ and $\widetilde{L}$ somehow to $T^*\Lambda$ so that the resulting Lagrangians satisfy the tangency condition. We cannot simply pullback by $p$ and then onto some tubular neighborhood of $T^* \Lambda|_S$. Instead, we begin by choosing a metric $g$ on $\Lambda$; such a metric defines a section $\psi$ of $p:T^*\Lambda|_S \to T^* S$ in the following way. For a point $(x,\phi) \in T^* S$, $\psi(x,\phi) = (x,\Phi)$ where $\Phi$ is the unique covector such that $\Phi|_S = \phi$ and it vanishes on the subspace $g$-orthogonal to $T_x S$.
			
			Next, with $k=\dim S$, if we pick a tubular neighborhood $\nu:N \to T^*\Lambda|_S$, then by construction $Q_s:=\nu^{-1}(\psi(\widetilde{Q}_s))$ is an $n$-dim Lagrangian of $T^*\Lambda$ where at each $x \in C$, its tangent space splits as $T_x Q_s = T_x S \oplus T_x F$. Here, $F$ is the fiber of $\nu$ and is $n-k$-dimensional.
			
			So we want to pick the tubular neighborhood in such a way that for each $x \in C$, the tangent space $T_x L = T_x S \oplus T_x F$. We already know that $L$ descends to $\widetilde{L}$ and at points $x \in C$, $T_xS = T_x \widetilde{L}$ so we simply need $T_x F \subset T_x L$.
			
			Now, for a point $x \in C \subset S \subset \Lambda$, $T_{(x,0)}T^*\Lambda \cong T_x S \oplus T_x S^\perp \oplus T^*_x \Lambda$. Here, $T_x S \oplus T_x S^\perp \cong T_x \Lambda$ (recall, we chose a metric on $\Lambda$). On the otherhand, $T_{(x,0)} T^*\Lambda|_S \cong T_x S \oplus T^*_x \Lambda$ and hence, the fibers of the normal bundle of $T^*\Lambda|_S$ at $x$ can be identified with $T_x S^\perp$. Because $T_x S^\perp \subset T_x \Lambda \subset T_{(x,0)} T^*\Lambda$ and $T_x \Lambda$ is a Lagrangian subspace, then $T_x S^\perp$ is isotropic. So we simply choose a tubular neighborhood of $T^*\Lambda|_S$ with fibers $F$ such that $T_x F = T_x S^\perp$. This is possible since $T_x S^\perp$ is complementary to $T_{(x,0)} T^*\Lambda|_S$. Therefore, by construction, $Q_s$ and $L$ are tangent along $C$. By Lemma \ref{isotopy}, we may conclude that near $C$ in $T^*\Lambda$, we also have a Hamiltonian family of vector fields $X_{K_s}$ generating the isotopy. Moreover, similar to in the symplectic reduction, because of the tangency of all the $Q_s$ along $C$, we may conclude that $\text{Hess}_x K_s = 0$ for $x \in C$. We had postponed showing this Hamiltonian isotopy result for Case 1 but the argument is exactly the same. In some sense, there aren't really two cases. Case 1 is simply the scenario where the symplectic reduction is trivial.
			
			Lastly, the family is $C^1$-small simply because we have bounds on the first derivative of $\rho$ and hence on the first derivative of $\rho_s$. On the other hand, $df_1$ vanishes on $C$ and hence, is small near $C$.
		\end{proof}
		
		Let $\Sigma$ be the lift of $\widetilde{\Sigma}$. It is the total space of some bundle over $\widetilde{\Sigma}$, produced by lifting twice. The first time by $p$, we would have the total space of a fibration over $\widetilde{\Sigma}$ with fiber $F$. And then we lift a second time by $\nu$. Each time, the fibers are contractible so $\Sigma$ deformation retracts to $\widetilde{\Sigma}$ which itself deformation retracts to $C$.
		
		In both Case 1 and Case 2, for $s < 1$, $Q_s \cap \Lambda = C$ and $Q_1 \cap \Lambda = \Sigma$. So we may fix an open neighborhood $V$ containing $\Sigma$ and that will suffice as an isolating neighborhood.
	\end{proof}
	
	\noindent \textbf{Remarks:}
	\begin{itemize}
		\item Similar to the ``Morse'' case, we call $\Sigma$ a \textbf{thickening}. We also have an analog of quasi-minimal degeneracy for Lagrangians.
		
		\item It is important to note that in Case 2, $L$ is not the graph of an exact 1-form. It is only when we pass to the symplectic reduction that we have a graph of an exact 1-form in the vicinity of $C$.
		
		\item For the last step of the proof, we also have the option of taking a similar approach to the alternative outlined in the proof of Lemma \ref{fperturb}. This would give us $\widetilde{\Sigma}$ itself as an isolated critical set rather than its lift. However, one technical but resolvable issue is that if we multiply the pullback function $\nu^* p^* \widetilde{K}_s$ by $(1+r^4)$ where $r$ is the radial function, the generated flow of the resulting Hamiltonian will not be tangent to the fibers. To briefly illustrate this, consider the simple example $K:\R \to \R, x \mapsto x^2$. If $y$ is the fiber coordinate of $T^* \R$, then $\check{K}:=(1+r^4)\pi^*K = (1+y^4)x^2$ and $X_{\check{K}} = 4x^2 y^3 \, \partial_x - 2x(1+y^4)\, \partial_y$. When $x \neq 0$ and $y \neq 0$, the vector field is not tangent to the fibers. See figure (when $x=0$, the vector field vanishes; this is not depicted in the image).
		
		\begin{center}
			\includegraphics[scale=0.4]{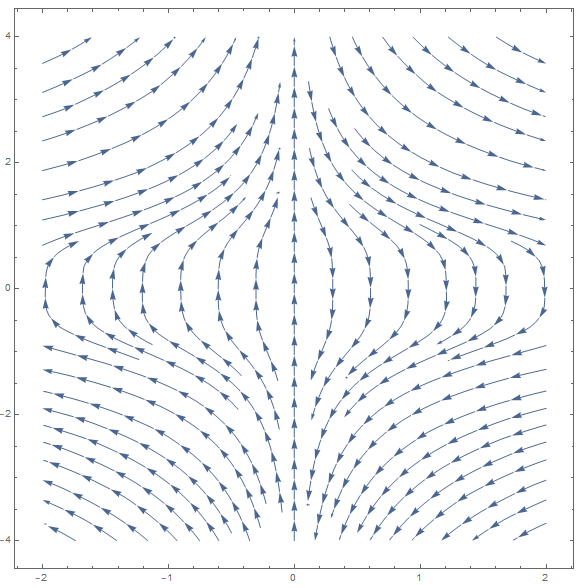}
			
			\caption{Made with Wolfram Mathematica 12.1}
		\end{center}
		
		However, in general, one can modify the Hamiltonian so that it coincides with the $\nu^* p^* \widetilde{K}_s$ when restricted to $Q_s$ and hence, the flow is tangent to the fibers along the family of Lagrangians.
	\end{itemize}
	
	Having defined flattened degeneracy and proven a thickening result, we now give a definition for quasi-minimal degeneracy.
	
	\begin{definition} \label{LMD}
		We say that $C \subset \Lambda \cap L$ is \textbf{quasi-minimally degenerate} (QMD) if there is a time-dependent Hamiltonian $K$ so that $C$ is an isolated family of intersection points of $\Lambda$ and $\phi^K_t(L)$ for each $t \in [0,1]$ and moreover, $C$ is flattened degenerate along $S$ with respect to $\Lambda, \phi^K_1(L)$. 
	\end{definition}
	
	Much like the ``Morse'' situation, we also have an isotopy which keeps $C$ as an isolated set in the intersections and at time 1, we require flattened degeneracy. Hopefully, context makes it clear whether we mean QMD in the ``Morse'' sense or in the Lagrangian sense. In the next section, we will show that there is a good reason for using the same name.
	
	\subsection{Relating the ``Morse'' and Lagrangian Definitions}
	
	In this section, we prove a result which relates the ``Morse'' definition of quasi-minimally degenerate to the Lagrangian definition of quasi-minimally degenerate.
	
	\begin{proposition} \label{graph}
		If $L$ is the graph of $df$ in $T^*M$ where $f$, as a function, is quasi-minimally degenerate in the sense of Definition \ref{QMD} along a critical locus $C \subset M$, then $C$ is also quasi-minimally degenerate in the Lagrangian sense of Definition \ref{LMD}.
	\end{proposition}
	
	To prove this proposition, we first make the following general observation. If $\pi:T^*M \to M$ is the cotangent bundle of $M$ and $\tau:M \to \R$ is any function, then $\pi^* \tau$ is a function on $T^*M$. Observe that since $d(\pi^*\tau) = d\tau \circ d\pi$, if $Y \in TF$ where $F$ is a fiber of the cotangent bundle (so $Y$ is in the vertical directions of $TT^*M$), then $0=d(\pi^* \tau)(Y) = \omega(X_{\pi^* \tau},Y)$ where $X_{\pi^* \tau}$ is the Hamiltonian vector field associated to $\pi^* \tau$. This holds for every $Y$ in the vertical direction. Hence, $X_{\pi^* \tau}$ is $\omega$-orthogonal to the fiber $F$. But the fiber is Lagrangian and so $TF = TF^\omega$. This means $X_{\pi^* \tau} \in TF$ as well. The main point is that the flow of $\pi^* \tau$ is tangent to the fibers of the bundle.

	Next, let $U \subset M$ be an open subset diffeomorphic to an open subset of $\R^n$ such that $T^*M|_U \cong U \times \R^n$; i.e. the cotangent bundle is trivialized on this set, and $U \times \R^n$ is a Darboux chart. This means that if we let $(x_1,...,x_n)$ be coordinates of $U$ and $(y_1,...,y_n)$ be fiber coordinates, then the symplectic structure of $U \times \R^n$ is symplectomorphic to $\omega_{st} = \sum^n_i dx_i \wedge dy_i$. 
	
	Next, let $\tau:U \to \R$ be a smooth function. Then, $X_{\pi^* \tau}$ can be computed. Indeed, as is well known in classical mechanics: \[X_{\pi^* \tau} = \sum^n_i \frac{\partial (\pi^* \tau)}{\partial y_i} \partial_{x_i} - \frac{\partial (\pi^* \tau)}{\partial x_i} \partial_{y_i}.\]
	
	We know $X_{\pi^* \tau}$ is tangent in the fiber directions; this is because $\pi^* \tau$ is constant on the fibers and hence the $\frac{\partial (\pi^* \tau)}{\partial y_i}$ vanish. On the other hand, the coefficients of the $\partial_{y_i}$ are simple $\frac{\partial \tau}{\partial x_i}$. This means that the flow is translation in the fibers by $-t \, d\tau$ where $t$ is the time. This shows us that for a function $\tau:M \to \R$, the image of $d\tau$ viewed as a section of $\pi:T^*M \to M$ is a Lagrangian which is mapped to the zero section by the time 1 flow of $\pi^* \tau$. We summarize this in a lemma.
	
	\begin{lemma} \label{cotflow}
		Let $\tau:M \to \R$ be a smooth function and $\pi:T^*M \to M$ be the cotangent bundle of $M$. Then the autonomous Hamiltonian vector field $X_{\pi^* \tau}$ is tangent to the fibers and the time $t$ flow $\phi^{\pi^* \tau}_t$ acts by translation on the fibers via $-t\, d\tau$.
	\end{lemma}
	
	\noindent \textit{Proof of Prop. \ref{graph}} In our situation, we assume that $f$ is quasi-minimally degenerate along $C$ with corresponding submanifold $S$ and that $L$ is the graph of $df$ inside of $T^*M$. By definition, there exists a $\tau:M \to [0,\infty)$ with several properties including $\tau^{-1}(0)=C$, $(f-\tau)|_S$ is minimal on $C$, and $\ker \text{Hess}_x (f-\tau) = T_x S$ for all $x \in C$. Though it's not necessary, we write this as two cases for the purpose of illustrating why we need that $\ker \text{Hess}_x \tau$ is transverse to $T_x S$ for $x \in C$. \\
	
	\noindent \textbf{Case 1:} Suppose $\dim S = \dim M$. In this case, if $\pi:T^*S \to S$ is the contangent bundle, then the flow of $K:=\pi^* \tau$ fixes the points $x \in C$. Moreover, $(f-\tau)|_S$ has $C$ as a minimum; WLOG, suppose the minimum value is 0. Because $\tau \geq 0$, for $x$ near $C$, $\tau(x) > 0$. In order for $C$ to be a minimum of $(f-\tau)|_S$, we need $f(x) > \tau(x)$. This implies that in a small neighborhood of $C$, $f$ increases more rapidly than $C$ in all directions of $S$ (and hence of $M$). In particular, for $x$ near but not in $C$, $df_x - t \, d\tau_x \neq 0$ for $t \in [0,1]$. This translation, as Lemma \ref{cotflow} tells us, is the flow of $\pi ^* \tau$. Hence, $C$ remains an isolated set of the intersection $M \cap \phi^K_t(L)$ for each $t$.
	
	Next, since $S$ is codim 0, $\text{Hess}_x (f-\tau)$ vanishes completely on $T_x M$ for $x \in C$. Conceptually, this means that the 1st order derivatives of $f-\tau$ are not deviating at all from zero along $C$ and hence $\phi^K_1(L)$ is tangent to $M$ (and hence $S$) at points in $C$; i.e. $T_x M = T_x S = T_x M \cap T_x \phi^K_1(L)$ which means $T_x M = T_x \phi^K_1(L)$ for $x \in C$. 
	
	Lastly, to show that $C$ is flattened degenerate with respect to $\phi^K_1(L), M$, we need a Hamiltonian $H$ such that $\phi^H_1(S) \subset \phi^K_1(L)$ along with the other properties. Since $\phi^K_t(L)$ can be described as the graph of $d(f-t\tau)$, we may simply let $H=\pi^*(\tau-f)$. Then, $dH|_C = 0$ and $\text{Hess}_x H = 0$ for $x \in C$. And since it is an autonomous Hamiltonian, $\frac{d}{dt}H =0$. The flow of $-H$ brings $\phi^K_1(L)$ to the zero section; hence, the flow of $H$ brings the zero section to $\phi^K_1(L)$, including $\phi^H_1(S)$. \\
	
	\noindent \textbf{Case 2:} Suppose $\dim S < \dim M$. In this case, one concern is that though $(f-\tau)|_S$ is minimal on $C$ and hence, $C$ is isolated in $S$, it may be that $K=\pi^* \tau$ will behave badly in the normal directions to $S$. This is why it is crucial that we have another condition on $\tau$: $\ker \text{Hess}_x \tau$ is transverse to $T_x S$ for $x \in C$. If we pick a metric, we can then consider the Hessian of $\tau$ at a point $y \in S$, near $C$. Since it is near $C$ and transversality is an open condition, $\ker \text{Hess}_y \tau$ is transverse to $T_y S$. In other words, $\tau$ doesn't do anything in the directions normal to $S$, such as introduce new critical points.
	
	Thus, we can still use $K=\pi^* \tau$ as the Hamiltonian and $C$ remains an isolated subset of $\phi^K_t(L) \cap M$ for all $t \in [0,1]$. Once again, $\phi^K_1(L)$ is the graph of $d(f-\tau)$. Moreover, since $\ker \text{Hess}_x (f-\tau) = T_x S$ for $x \in C$, we have that $T_x S = T_x M \cap T_x \phi^K_1(L)$. And lastly, the same $H$ still has all the correct properties. \qed \\
	
	\noindent The following corollary is immediate.
	
	\begin{corollary}
		Let $\Lambda,L \subset (M,\omega)$ be Lagrangians and $C \subset \Lambda \cap L$. Suppose there exists a Weinstein neighborhood $U$ around $C$ and a symplectomorphism $\varphi:U \to V$ where $V$ is a neighborhood of the zero section of $T^* \Lambda$, $\Lambda \cap U$ is mapped to the zero section, and $L \cap U$ is mapped to a Lagrangian in $V$ which is the graph of $df$ where $f$ is QMD. Then, $C$ is QMD in the Lagrangian sense.
	\end{corollary}
	
	\subsection{Local Lagrangian Floer Theory}
	
	In the definitions above, we took the effort to ensure that if $C$ is an isolated subset of a Lagrangian intersection, then whenever we perturbed the Lagrangians by Hamiltonian isotopies, $C$ remained isolated or at least its homotopy type does not change. The reason for this is because we want to study Lagrangian intersections using Lagrangian Floer theory. Let's begin with a rough intuitive description of the usual Lagrangian Floer homology before discussing a local homology theory.
	
	In nice cases, Lagrangian Floer homology associates to a pair $(L_0,L_1)$ of Lagrangians a group $HF^*(L_0,L_1)$ which has a few properties.
	
	\begin{enumerate}
		\item $HF^*(L_0,L_1)$ ``categorifies'' intersection numbers in the sense that $\chi(HF^*(L_0,L_1)) = L_0 \cdot L_1$ (intersection number of the smooth topology).
		
		\item $HF^*(L_0,L_1)$ is Hamiltonian isotopy invariant. So if $\phi^{H_0}, \phi^{H_1}$ are Hamiltonian diffeomorphisms, then $HF^*(\phi^{H_0}(L_0),\phi^{H_1}(L_1)) \cong HF^*(L_0,L_1)$.
		
		\item If $L=L_0=L_1$, then $HF^*(L,L) \cong H^*(L)$, the singular homology of $L$.
		
		\item If $L_0$ and $L_1$ intersect transversally, then $HF^*(L_0,L_1)$ is the homology of a chain complex generated by the intersection points. This implies that the rank of $HF^*$ gives a (refined) lower bound for Lagrangian intersections: $\#(\phi^H_1(L_0) \cap L_1) \geq \text{rk} \, HF^*(L_0,L_1) \geq L_0 \cdot L_1$.
	\end{enumerate}
	
	We are deliberately vague about what ``nice cases'' means but these properties cannot always hold. Indeed, a compact Lagrangian $L$ in $\C^n$ can be displaced by a Hamiltonian $\phi$ so that $L \cap \phi(L) = \varnothing$. Hence, $0 \geq \text{rk} \, HF^*(L,\phi(L)) =\text{rk} \, HF^*(L,L) = \text{rk} \, H^*(L)$; this obviously cannot happen. The example can be modified so that $L$ is displaced by a compactly supported Hamiltonian isotopy and hence, the compact manifold $\mathbb{CP}^n$ also serves as a counterexample.
	
	However, moving forward, we will not be too concerned with these issues as they have been addressed in many other texts. We shall simply proceed to the local situation and add some rigor to the description. For full details, consult section 3 of \cite{Pozniak}. Let $P(L_0,L_1)$ be the space of paths starting on $L_0$ and ending on $L_1$. Let $\mathcal{U} \subset P(L_0,L_1)$ be a closed subset and let $ev:\mathcal{U} \times I \to M$ be the map sending $(\gamma,t) \mapsto \gamma(t)$. We say that $\mathcal{U}$ is bounded if the image of $ev$ is precompact. 
	
	Next, we would like to define an action functional. In general, the action functional is defined only on the universal cover $\widetilde{P}(L_0,L_1)$. Choose a base point $\gamma_0 \in P(L_0,L_1)$. Let $u:I \times I \to M$ represent an element $\tilde{\gamma} \in \widetilde{P}$; i.e. $u(0,t) = \gamma_0(t), u(1,t) = \gamma(t)$, and $u(s,i) \in L_i$. We also introduce a Hamiltonian $H: M \times [0,1] \to \R$. Then, the action functional is
	
	\[A_H(\gamma) := \int u^* \omega + \int^1_0 H_t(\gamma(t)) \, dt. \]
	
	\noindent The critical points are precisely the paths $\gamma(t) \in P(L_0,L_1)$ satisfying $\gamma'(t) = X_t(\gamma(t))$. Here, $X_t$ is the time-dependent Hamiltonian vector field. To make this a local theory, we simply consider the critical points of $A_H$ inside of $\mathcal{U}$.
	
	We also consider a moduli space $\mathcal{M}_{J,H}(L_0,L_1,\mathcal{U})$ of $J$-holomorphic strips $u:\R_s \times I_t \to M$; these strips satisfy a twisted Cauchy-Riemann equation $\overline{\partial}u + \nabla_u H = 0$. Additionally, we require the elements of this moduli space to satisfy $u(s,\cdot) \in \mathcal{U}$ for all $s \in \R$. The \textbf{maximal invariant subset} $\mathcal{S}_{J,H}(\mathcal{U})$ is defined to be the image of $\mathcal{M}_{J,H}(L_0,L_1,\mathcal{U})$ under the evaluation map $ev: \R \times \mathcal{M}_{J,H}(L_0,L_1) \to P(L_0,L_1)$, $ev(s,u)(t) = u(s,t)$.
	
	We also say that $\mathcal{S}_{J,H}(\mathcal{U})$ is isolated if its closure under the compact-open topology is contained in the interior of $\mathcal{U}$. If it is isolated, then whenever a sequence $u_n \in \mathcal{M}_{J,H}(L_0,L_1,\mathcal{U})$ converges to $u \in \mathcal{M}_{J,H}(L_0,L_1)$, in fact, $u \in \mathcal{M}_{J,H}(L_0,L_1,\mathcal{U})$.
	
	It was shown in Pozniak's thesis \cite{Pozniak} that:
	
	\begin{proposition} \label{lag_invariance}
		Assume that $\mathcal{U}$ is bounded, $\mathcal{S}_{J,H}(\mathcal{U})$ is isolated and the symplectic action $\mathcal{A}_H$ is defined on $\mathcal{U}$. There is an $\epsilon > 0$ such that if $\|J'-J\|_{C^1} < \epsilon$ and $\|H'-H\|_{C^1} < \epsilon$, then $\mathcal{S}_{J',H'}(\mathcal{U})$ is also isolated. Moreover, if both pairs $(J,H), (J',H')$ are regular, then
		\[H_*(C_*(L_0,L_1,\mathcal{U},J',H')) \cong H_*(C_*(L_0,L_1,\mathcal{U},J,H)).\]
	\end{proposition}
	
	\noindent In this context we say $\mathcal{S}_{J',H'}(\mathcal{U})$ is a continuation of $\mathcal{S}_{J,H}(\mathcal{U})$. When the context is clear and we have the given Lagrangians, $H$, and $J$, we may sometimes simplify notation and just write $HF(\mathcal{U})$ for the local Floer homology.
	
	\subsection{Hamiltonian Floer Theory} \label{hamfloer}
	
	It should be stated that we may define flattened and quasi-minimal degeneracy for Hamiltonian Floer theory as well since we can recover Hamiltonian Floer theory from Lagrangian Floer theory. If $K:M \times [0,1] \to \R$ is a time dependent Hamiltonian and $\phi^K_1$ is its time one flow, then we study the fixed points of $\phi^K_1$. Let $\Gamma \subset M \times M$ be the graph of $\phi^K_1$ and $\Delta \subset M \times M$ the diagonal. The fixed points of $\phi^K_1$ are in one-to-one correspondence with the intersection points $\Gamma \cap \Delta$ (for details of how to relate the differentials of the two complexes, see \cite{ham_lag}). In the symplectic manifold $(M \times M, \omega \oplus (-\omega))$, $\Gamma$ and $\Delta$ are Lagrangian submanifolds. Hence, we say that a set $C$ of fixed points of $\phi^K_1$ is flattened degenerate (QMD resp.) if and only if $C \times C$ is flattened degenerate (QMD resp.) in $\Gamma \cap \Delta$. We conjecture that there is an equivalent definition that is more natural or at least easier to work with for the Hamiltonian setting but we do not yet have good candidates. As a suggestion, if $K$ is the time-dependent Hamiltonian from above, then the definition should involve studying $\ker(\phi^K_1 - \id)$.
	
	\section{Clean Intersections and Pozniak's Results}
	
	From now on, we shall prefer the Lagrangian viewpoint but everything translates over to the Hamiltonian viewpoint as outlined above. In Pozniak's thesis \cite{Pozniak}, he gives us a way to compute local Floer homology for Lagrangians that cleanly intersect along a submanifold.
	
	\begin{definition}
		Let $L_0,L_1$ be Lagrangians and $N$ a submanifold. We say that $L_0$ and $L_1$ have a clean intersection along $N$ if $N \subset L_0 \cap L_1$ and for every $x \in N$, $T_x N = T_x L_0 \cap T_x L_1$.
	\end{definition} 
	
	In general, the intersection may be wild but if some part of the intersection is clean along $N$ and there are no other intersection points in a small neighborhood of $N$, there is a way to define local Floer homology in a neighborhood of $N$. If $\partial N \neq \varnothing$, then we do not have a clean intersection but it is straightforward to adapt Pozniak's arguments to intersections along manifolds with boundary; we will do that in section 5. We first record Pozniak's original results.
	
	\subsection{A Standard Model}
	
	The first of Pozniak's results shows that cleanly intersecting Lagrangians have a standard model.
	
	\begin{theorem} \label{normal_form}
		Let $(M,\omega)$ be a symplectic manifold and $L_0,L_1$ two Lagrangian submanifolds of $M$ which intersect cleanly along a compact manifold $N$. There exist a vector bundle $\tau:L \to N$, a neighborhood $V_0$ of $N$ in $T^*L$, a neighborhood $U_0$ of $N$ in $M$, and a symplectomorphism $\phi:U_0 \to V_0$ such that
		\[\phi(L_0 \cap U_0) = L \cap V_0 \text{ and } \phi(L_1 \cap U_0) = TN^{ann} \cap V_0.\]
	\end{theorem}
	
	\noindent Before sketching Pozniak's proof, here is a short outline. First, he proved that cleanly intersecting Lagrangians may be put into a standard form. Then, by way of a Moser-type argument, Pozniak showed that there exists a vector bundle $\tau:L \to N$ and also neighborhoods $U$ of $N$ in $M$ and $V$ of $N$ in $T^*L$, and a symplectomorphism $\phi:U \to V$ which satisfies:
	$\phi(L_0 \cap U) = L \cap V$ and $\phi(L_1 \cap U) = TN^{ann} \cap V$. Here, $TN^{ann} = \{\alpha \in T^*L_N: \alpha|_{TN} = 0 \}$ (the annihilator). The $L$ he chose is $L=TN^\perp \subset TL_0$ for a chosen metric on $L_0$ and the exponential map gives the desired tubular neighborhood. The proof does not actually rely on compactness of $N$ and can be adapted to open manifolds. \\
	
	\noindent \textit{Sketch of Pozniak's proof:}
	
	\begin{enumerate}
		\item Use the Weinstein neighborhood theorem to view a neighborhood of $L_0$ as symplectomorphic to a neighborhood of the zero section of $T^*L_0$. Choose a metric on $L_0$ and let $L = TN^\perp \subset TL_0$. The exponential map gives a diffeomorphisms of neighborhoods of $N$ in $L_0$ and $L$ which induces a symplectomorphism. Therefore, without loss of generality, treat $L_0 = L$ and $M = T^*L$. 
		
		\item Let $L_2 = TN^{ann}$. The goal now is to show there exists a symplectomorphism $\chi_1:U_1 \subset T^* L \to V_1 \subset T^* L_2$ where both the domain and range are neighborhoods of $N$ such that
		\begin{enumerate}
			\item $\chi_1|_{L_2 \cap U_1} = \id$ 
			
			\item $\chi_1(L \cap U_1) \subset T^*L_2|_N$
			
			\item $\chi_1(L_1 \cap U_1) = \Gamma_\alpha$, the graph of a 1-form $\alpha$ on $L_2$.
		\end{enumerate}
		\item Assuming $\chi_1$ exists, note that $N \subset \Gamma_\alpha$ which implies that $\alpha_N = 0$. Then for $x \in L_2 \cap U_1$, the map 
		\[\psi_\alpha:T^*_x L_2 \to T^*_x L_2, \beta \mapsto \beta - \alpha(x) \]
		
		is a symplectomorphism. We may choose a sufficiently small neighborhood $V_2$ of $N$ in $T^*L_2$ so that $\psi_\alpha(V_2) \subset \chi_1(U_1)$.
		
		\item Let $\phi:U_0 \to V_0$ be defined by $\phi(x) = \chi^{-1}_1 \circ \psi_\alpha \circ \chi_1(x)$. Letting $U_0 = \chi^{-1}_1(V_2)$, we can check that $\phi$ satisfies each of the properties we want.
	\end{enumerate}
	
	So now, we need to show that $\chi_1$ exists.
	
	\begin{enumerate}
		\item Note that we need only show that for a defined map $\chi_1$, $\chi_1(L_1)$ should be transverse to the fibers of $T^* L_2$ in order for the image to be a graph.
		
		\item Let $E = \ker d \tau$ be the vertical subbundle of $TL$ (recall that $L = TN^\perp$). For $x \in N$,
		\[T_x(T^*L) = T_x L \oplus T^*_x L = E_x \oplus T_x N \oplus T_x N^{ann} \oplus E^{ann}_x. \]
		
		\item It is straightforward to show that $T_x L_1 \cap (E_x \oplus E^{ann}_x) = 0$. Then, $L_1$ is transverse to $E^{ann}_x$. If show $\chi_1(E^{ann} \cap U_1) \subset T^*_x L_2$, then we'll have shown that $\chi_1(L_1)$ is transverse to the fibers of $T^* L_2$.
		
		\item We need two lemmas:
		\begin{enumerate}
			\item There exists a vector bundle $\sigma:T^*L \to TN^{ann}$ with fibers as Lagrangian submanifolds of $T^*L$. In particular, $\sigma^{-1}(x) = E^{ann}_x$ for all $x \in N$.
			
			The proof mainly involves checking that a proposed $\sigma$ does have a vector bundle structure.
			
			\item Let $\sigma:V \to L$ be a vector bundle such that $(V,\omega_0)$ is a symplectic manifold and the fibers $V_x$ are Lagrangian. Then for every compact set $K \subset L$, there is a fiber preserving symplectomorphism $\chi$ defined in a neighborhood of $U$ of $K$ in $V$; $\chi:(U,\omega_0) \to (T^*L,\omega)$ where $\omega$ is the standard symplectic form on $T^*L$. Moreover, $\pi_L \circ \chi = \sigma|_U$ and $\chi|_L = \id$.
			
			The proof involves a Moser-type argument to show the fiber preserving property.
			
		\end{enumerate}
		\item The two lemmas immediately show the existence of $\chi_1$. $\chi|_L = \id$ gives (a), $\pi_L \circ \chi = \sigma|_U$ gives (b), and the fiber preserving property gives (c). \qed
	\end{enumerate}
	
	\subsection{Morse and Floer Data Coincide}
	
	The second result of Pozniak's results shows that a $C^1$ small Morse function allows us to identify Morse and Floer critical points and flow lines.
	
	\begin{theorem} \label{Morse}
		Let $(N,g_N)$ be a compact, Riemannian manifold, $\tau:L \to N$ a vector bundle over $N$ and $f:N \to \R$ a $C^2$ function on $N$. Let $\pi:T^*L \to L$, $f_L = f \circ \tau$, and $H = f \circ \tau \circ \pi$. We can construct a metric $g$ on $L$ by lifting $g_N$ (see p. 81-82 for details). Let $J = J_g$ be the associated almost complex structure defined using $d\lambda$ and $g$ ($\lambda$ is the canonical 1-form on $T^*L$). We also suppose there is a neighborhood $U$ of $N$ in $L$ such that $\|\nabla^g df_L(x)\| \leq 1$ for all $x \in U$. Then the following holds:
		
		\begin{enumerate}
			\item All critical points and gradient lines with respect to $J$ for the action functional $\mathcal{A}_H$ in $\Omega(\pi^{-1}(U), U, TN^{ann})$ are $t$-independent and so they are in 1-1 correspondence with the critical points and the gradient lines of $f$ with respect to $g_N$.
			
			\item The critical points of $\mathcal{A}_H$ are nondegenerate if $f$ is a Morse function. In this case, if $x^\pm \in Crit(f)$ and $u:\R \to N$ is a t-independent element of $\mathcal{P}(x^-,x^+)$, then the linearized operator $D_{J,H}(u)$ is onto if and only if the operator $D_f(u): W^{1,p}(u^* TN) \to L^p(u^*TN)$, $D_f(u) \xi = \nabla_s \xi + \nabla_\xi \nabla f(u)$ is onto and the assignment $\xi \mapsto \xi'(s,t) = \xi(s)$ gives the isomorphism $\ker D_f(u) \cong \ker D_{J,H}(u)$.
		\end{enumerate}
	\end{theorem}
	
	\noindent \textbf{Remark:} Note that once a metric and function are fixed on $N$, Pozniak gives a \textit{specific} metric and almost complex structure on $T^*L$, rather than take generic pairs. Despite the non-genericity, the second part of the result asserts that we still have smooth moduli spaces. \\
	
	\noindent \textit{Sketch of Pozniak's proof:}
	
	\begin{enumerate}
		\item There exists local coordinates $x = (q,q',p,p')$ on $T^*L$ such that $\frac{\partial H}{\partial p} = \frac{\partial H}{\partial p'} = \frac{\partial H}{\partial q'}=0$ and $X_H(0,0,df(q),0)$. So the Hamiltonian flow is $\phi_t(q,q',p,p') = (q,q',p+tdf(q),p')$. 
		
		Now consider paths $\gamma$ with boundary conditions $\gamma(0) \in L$ and $\gamma(1) \in TN^{ann}$. When $x \in L$, $p=p'=0$ and when $x \in TN^{ann}$, $q' = p + df(q) = 0$. Hence, the only Hamiltonian paths are constant: $x(t) = (q,0,0,0)$ with $q$ being a critical point of $f$.
		
		\item Suppose that $D^2f(q) :=\text{Hess}_q f$ is nondegenerate. Then in these coordinates,
		\[D \phi_1(x) =
		\begin{pmatrix}
			I & 0 & 0 & 0 \\
			0 & I & 0 & 0 \\
			D^2 f(q) & 0 & I & 0 \\
			0 & 0 & 0 & I
		\end{pmatrix} \]
		
		Let $v = (Q,Q',0,0)$ be a vector tangent to $L$. Then, $D \phi_1(x) v = (Q,Q',D^2 f(q) Q, 0)$ is tangent to $TN^{ann}$ if and only if $Q' = D^2f(q) Q = 0$. In this case, $Q = 0$ as well as $D^2 f(q)$ is nondegenerate. So $Q = Q' = 0$. Hence, $D \phi_1(x) (T_x L) \cap T_x TN^{ann} = \{0\}$. Therefore, $x$ is nondegenerate as a critical point of the action functional $\mathcal{A}_H$.
		
		\item Let $g^D$ be the Kaluza-Klein metric on $T^*L$ which is a ``diagonal'' lift of $g$. The important feature of $g^D$ is that it is compatible with the canonical symplectic structure and $J$ that we've defined. Let $V = \ker d \pi$ be the vertical subbundle of $T(T^*L)$. Then $dH = df_L \circ d \pi$ vanishes on $V$ which means $\nabla H$ with respect to $g^D$ is in the horizontal subspace: $\nabla H(\xi) \in H_\xi$.
		
		Moreover, $d \pi |_{H_\xi}$ is an isometry so $d \pi (\nabla H(\xi)) = \nabla^g f_L(\pi(\xi))$ which tells us that $\nabla H$ is a horizontal lift of $\nabla^g f_L$. Thus, if $x \in N$, then $\nabla H(x) = \nabla^(g_N) f(x) \in TN$. This means that if $u:\R \to N$ is a gradient line of $f$, then $v(s,t) := u(s)$ is a gradient line of $\mathcal{A}_H$ satisfying our boundary conditions. 
		
		\item Suppose now that $v:\R_s \times I_t \to U$ satisfies the Floer equation. We want to show that $v$ is $t$-independent. Let $x(s,t)$ and $y(s,t)$ be the horizontal and vertical components of $v(s,t)$.  Then $\frac{\partial v}{\partial s}$ and $\frac{\partial v}{\partial t}$ also decompose into horizontal and vertical components, giving a new form of the Floer equation:
		\[\frac{\partial x^*}{\partial s} - \nabla_t y + dH_1(x) = 0, \hspace{5mm} \nabla_s y + \frac{\partial x^*}{\partial t} = 0. \]
		
		Here, $*$ means the dual using the metric $g$. Of course, there are also boundary conditions for these equations. We wish to show that $y \equiv 0$ and hence, $\frac{\partial x}{\partial t}$ which implies that $v(s,t) = x(s):\R \to N$ is a gradient line of $f$.
		
		\item Define 
		\[\gamma(s) := \frac{1}{2}\int^1_0 |y(s,t)|^2 \, dt. \]
		
		Note that $\lim_{s \to \pm \infty} \gamma(s) = 0$ and also $\gamma \geq 0$. So it attains a maximum on $\R$. However, Pozniak showed the following lemma: if $\|\nabla^g f_L \|_{L^\infty} < 1$, then $\gamma''(s) \geq 0$ for all $s \in \R$. When the hypothesis holds, this means that $\gamma$ is both concave up everywhere but also achieves a maximum. This implies that $\gamma$ must be constant and in fact, $\gamma \equiv 0$ because it limits to 0. Hence $y \equiv 0$.
		
		The proof of this lemma requires the crucial fact: $y$ is a solution to the elliptic equation $\Delta y - \langle \nabla df_L, \nabla_s y \rangle = 0$.
		
		\item To show $\ker D_f (u) \cong \ker D_{J,H}(u)$, we similarly decompose a vector field $\xi=(\zeta,\eta) \in \Gamma(u^* T(T^*L))$ into horizontal and vertical components and obtain a way of writing the linearized Floer equation in these horizontal and vertical components.
		
		We may show that $\zeta \in \ker D_f(u)$ is a solution to the linearized Floer equation and hence in $D_{J,H}(u)$. Conversely, let $\xi = (\zeta,\eta) \in D_{J,H}(u)$. If we similarly define
		\[\gamma_1(s) := \frac{1}{2}\int^1_0 |\eta(s,t)|^2 \, dt, \]
		
		we may prove as Pozniak did that when $\|\nabla^g df_L \|_{L^\infty} < 1$, then $\gamma_1''(s) \geq 0$ for all $s$. Hence, $\eta \equiv 0$ and $\xi = (\zeta,0)$ is in $\ker D_f(u)$. 
		
		\item Lastly, similar arguments show that $\ker D^*_f(u) \cong \ker D^*_{J,H}(u)$ and so $D_f(u)$ is onto if and only if $D_{J,H}(u)$ is onto. \qed
		
	\end{enumerate}
	
	\subsection{Pozniak's Main Theorem}
	
	Finally, we state the main result from Pozniak's thesis about the local Floer homology of clean intersections.
	
	\begin{theorem} \label{Pmain_thm}
		Let $(M,\omega)$ be a symplectic manifold and $L_0,L_1$ two Lagrangian submanifolds that intersect cleanly along a compact, connected submanifold $N$. Fix base point $x_0 \in N$. If $U$ is any relatively compact neighborhood of $N$ such that
		
		\begin{enumerate}
			\item Aside from those in $N$, there are no other critical points of $\mathcal{A}$ in the connected component $P(U,L_0,L_1,x_0)$ of the constant path $x_0$ in the path space $P(U,L_0,L_1)$.
			
			\item The action function of $\omega$ is well-defined in $P(U,L_0,L_1,x_0)$, meaning, we do not need to lift to the universal cover.	
		\end{enumerate}
		
		Then $\mathcal{U} = P(U,L_0,L_1,x_0)$ is an isolating neighborhood and $\mathcal{S}_{J,0}(\mathcal{U}) = N$ for any almost complex structure $J$. There exists an almost complex structure $J_0$ and a Hamiltonian $H_0:M \to \R$ such that
		
		\begin{enumerate}
			\item $\mathcal{S}_{J_0,H_0}(\mathcal{U})$ is a continuation of $N$.
			
			\item $(J_0,H_0)$ is a regular pair and if $g_N = g_J|_N$, $f = H_0|_N$, then $(g_N,f)$ is Morse-Smale.
			
			\item The Floer complex $CF_*(\mathcal{U},J_0,H_0)$ coincides with the Morse complex $CM_*(N,g_N,f)$ and thus
			\[HF_*(\mathcal{U},\Z_2) \cong H^{sing}_*(N,\Z_2).\]
		\end{enumerate}
	\end{theorem}
	
	For the main theorem, Theorem \ref{Pmain_thm}, Pozniak first assumes we're in the setting of the standard form for a clean intersection. His theorem \ref{Morse} shows that for a $C^1$ small Morse function $f$, under a canonical metric $g$, almost complex structure $J$, and Hamiltonian $H$ on $T^*L$ defined from $f$, the Floer and Morse critical points and flow line all live within a small enough neighborhood $U$ of $N$ and they all coincide.
	
	Hence, the proof of the main theorem is mainly showing that for any chosen isolated neighborhood $U$ of $N$ which sits inside the neighborhood of our standard form for the clean intersection, there exists $\epsilon$ such that when $\|H\|_{C^1} < \epsilon$, the Floer critical points and trajectories are contained within $U$. \\
	
	\noindent \textit{Sketch of Pozniak's proof:} Given an isolated neighborhood $U$, suppose that there is no such $\epsilon$; that is, there is a positive sequence $\epsilon_n \to 0$ and Hamiltonians $H_n$ satisfying $|H_n|_{C^1} < \epsilon_n$ such that for each $n$, there is some Floer trajectories $u_n$ that leaves the neighborhood $U$.
	
	As a reminder, $u_n$ satisfies:
	\[\frac{\partial u_n}{\partial s} + J_n(u_n) \frac{\partial u_n}{\partial t} + \nabla H_n(u_n) = 0 \] 
	
	\noindent plus some boundary and limiting conditions. Then we may write the energy as
	
	\[E(u_n) = \mathcal{A}_0(x^-_n) - \mathcal{A}_0(x^+_n) + \int^1_0 H_n(x^+_n) - H_n(x^-_n) \, dt \]
	
	Pozniak proves a useful lemma 3.4.5: $|\mathcal{A}_0(\gamma)- \mathcal{A}_0(x)| \leq \|\dot{\gamma}\|^2_{L^2}$. With the lemma, he showed that $\|\dot{x}^+_n\| \leq \|H_n\|_{C^1} \leq \epsilon_n$ which bounds the first term. Similarly, there is a bound for the second term. The terms within the integral are bounded above by the $C^0$ norms of $H_n$ which are in turn, also bounded by $\epsilon_n$. Therefore, $E(u_n) \to 0$. This shows that $u_n \to u_0 \equiv const$. Moreover, $u_0$ must be in $N \subset U$. This contradicts the fact that the $u_n$ all leave the neighborhood $U$ at some point. \qed
	
	\section{Adapting Pozniak's Results} \label{adapt}
	
	For manifolds with boundary, it is customary to define the tangent space in such a way that even on the boundary, the tangent spaces are the same dimension as the interior. Thus, we'll keep the same definition: 
	
	\begin{definition}
		Two Lagrangians $L_0,L_1$ intersect cleanly along a submanifold with or without boundary $\Sigma$ if for all $x \in \Sigma$, $T_x \Sigma = T_x L_0 \cap T_x L_1$. 
	\end{definition}
	
	\noindent We now state a generalization of Theorem \ref{normal_form}. 
	
	\begin{theorem} \label{main_thm}
		Let $(M,\omega)$ be a symplectic manifold and $L_0,L_1$ two Lagrangian submanifolds. Suppose that $C \subset L_0 \cap L_1$ is a Lagrangian quasi-minimally degenerate set. There exists a perturbation for $C$ resulting in a submanifold $\Sigma$ with boundary which deformation retracts to $C$. Fix base point $x_0 \in \Sigma$. If $U$ is any relatively compact neighborhood of $\Sigma$ such that
		
		\begin{enumerate}
			\item There are no critical points of $\mathcal{A}$ other than those in $\Sigma$ in the connected component $P(U,L_0,L_1,x_0)$ of $x_0$ in the path space $P(U,L_0,L_1)$.
			
			\item The action function of $\omega$ is well-defined in $P(U,L_0,L_1,x_0)$.	
		\end{enumerate}
		
		\noindent Then $\mathcal{U} = P(U,L_0,L_1,x_0)$ is an isolating neighborhood and $\mathcal{S}_{J,0}(\mathcal{U}) = \Sigma$ for any almost complex structure $J$. There exists an almost complex structure $J_0$ and a Hamiltonian $H_0:M \to \R$ such that
		
		\begin{enumerate}
			\item $\mathcal{S}_{J_0,H_0}(\mathcal{U})$ is a continuation of $\Sigma$.
			
			\item $(J_0,H_0)$ is a regular pair and if $g_\Sigma = g_J|_\Sigma$, $f = H_0|_\Sigma$, then $(g_\Sigma,f)$ is a Morse-Smale.
			
			\item The Floer complex $CF_*(\mathcal{U},J_0,H_0)$ coincides with the Morse complex $CM_*(\Sigma,g_\Sigma,f)$ and thus
			\[HF_*(\mathcal{U},\Z_2) \cong H^{sing}_*(\Sigma,\Z_2).\]
		\end{enumerate}
	\end{theorem}
	
	\noindent \textbf{Remark:} Since Pozniak shows $CF_*(\mathcal{U},\Z_2)$ coincides with the Morse complex $CM_*(\Sigma,g_\Sigma,f)$ and since Morse theory works also for open submanifolds, we find that $HF_*(\mathcal{U},\Z_2) \cong H^{sing}_*(\Sigma,\Z_2)$ and in particular, $HF_k(\mathcal{U},\Z_2) = 0$ where $k = \dim \Sigma$.
	
	\begin{proof}
		The brunt of the work falls to Theorem \ref{perturb} which gives the perturbation yielding a thickening $\Sigma$ that deformation retracts onto $C$. So we just need to work with $\Sigma$. In the proof of Theorem \ref{normal_form}, $N$ is boundaryless and Pozniak puts a metric on $L_0$ and lets $L = TN^\perp$. We cannot directly follow suit because our $\Sigma$ has boundary and therefore, $L := T\Sigma^\perp$ would also have boundary. As a result, there cannot be a boundaryless neighborhood of $\Sigma$ in $T^*L$ symplectomorphic to a neighborhood of $\Sigma$ in $M$.
		
		Instead, we first take a slightly smaller perturbation of the QMD set to obtain a submanifold with boundary $\Sigma'$ inside of $\Sigma$. We have that $\partial \Sigma' \cap \partial \Sigma = \varnothing$. Furthermore, when the perturbation is sufficiently small, the local Floer homology is unaffected.
		
		Next, we take the interior of $\Sigma'$ to obtain an open manifold $\widehat{\Sigma}$ and extend the vector bundle structure to $\widehat{\Sigma}$. We then take $L = T\widehat{\Sigma}^\perp$. A similar construction gives us a desired $L_2 = T\widehat{\Sigma}^{ann}$.
		
		Therefore, we have a version of the standard form of clean intersections for an open manifold. Moreover, Pozniak's results do not really rely on the whether $\Sigma'$ is compact so we may use his arguments. For instance, one can do Morse theory on open manifolds. If there's any concern about the Floer data being pathological near the boundary of $\Sigma'$, we can simply take an even smaller perturbation before taking the interior. This never changes the homotopy type and hence, the Morse theoretic data is unchanged.
		
		Thus, to broaden Pozniak's main theorem to the case of minimally degenerate intersections, one can follow his arguments almost verbatim. The Morse and Floer data will coincide and the theorem follows.
	\end{proof}
	
	\section{A Spectral Sequence}
	
	From the results of Pozniak, onef can draw out a spectral sequence analogous to the Morse-Bott spectral sequence. Indeed, in a paper by Paul Seidel \cite{SeidelKnot}, he formulates Pozniak’s result in spectral sequence terms. 
	
	Suppose that $L_0, L_1$ are two Lagrangians with intersection decomposed into $\bigsqcup C_p$ where $C_1,... ,C_r$ are the connected components and are submanifolds. By Pozniak's results, there exist disjoint neighborhoods $U_p$ of the $C_p$ for a Hamiltonian $H$ that is sufficiently $C^1$ small. We may patch together an almost complex structure $J$ from the local data. Once done, each neighborhood $U_p$ has the property that all the relevant Floer theoretic data for $(H,J)$ stay within the $U_p$.
	
	Therefore, there exists a filtration on $CF_*(H)$ induced by the action functional. The filtration is preserved by the differential $\partial$. In more detail, let $x_p \in C_p$ and $a_p = \mathcal{A}(x_p)$ where we now view $x_p$ as a constant path. The $C_p$ are ordered so that $a_1 \leq a_2 \leq ... \leq a_r$. Then, there is a filtration on $CF$ by $CF^p$ which is generated by the critical points inside of $U_1 \cup...\cup U_p$. Clearly, $CF^p/CF^{p-1} = CF(U_p)$ and the homology there is the local Floer homology $HF(L_0,L_1,U_p)$.
	
	This provides a summary for how the usual Floer homology is related to the local Floer homology. One obtains a spectral sequence using this action filtration which converges to the usual Floer homology $HF_*(L_0,L_1)$ with the $E^1$-page being $E^1_{p,q} = HF_{p+q}(L_0,L_1,U_p)$. The local Floer homology is related to singular homology, albeit with some shift in grading.
	
	\subsection{Gradings}
	
	To discuss the shift in the gradings, we'll first recall some facts about the Maslov index. For more detail, one may consult Section 4 of Seidel's paper \cite{SeidelKnot} or the Robbin-Salamon papers that Seidel references \cite{RobbinSalamon1},\cite{RobbinSalamon2}. We will mostly stick to Seidel's notation.
	
	Let $\mathcal{L}(n)$ denote the Lagrangian Grassmannian for $(\R^{2n},\omega_0)$ and consider paths $\gamma,\gamma':I=[0,1] \to \mathcal{L}(n)$. The Maslov index $\mu(\gamma,\gamma')$ assigns values in $\frac{1}{2}\Z$ to these two paths and has some basic properties. Of primary importance to us are:
	
	\begin{enumerate}
		\item $\mu(\gamma,\gamma')$ depends on $\gamma,\gamma'$ only up to homotopy with fixed endpoints.
		
		\item $\mu$ is unchanged if one conjugates both $\gamma$ and $\gamma'$ by a path $\Psi: [0,1]\to Sp(2n, \R)$.
		
		\item $\mu$ is additive under concatenation of paths.
		
		\item $\mu(\gamma,\gamma')$ vanishes if the dimension of $\gamma(t) \cap \gamma'(t)$ is constant.
		
		\item $\mu(\gamma,\gamma') \equiv \frac{1}{2} \dim(\gamma(0) \cap \gamma'(0)) - \frac{1}{2} \dim(\gamma(1) \cap \gamma'(1)) \pmod{1}$.
	\end{enumerate}
	
	Letting $\gamma_x$ denote the constant path at a point $x$, choose a path $I \to P(L_0,L_1)$ which begins at $\gamma_{x_-}$ and ends at $\gamma_{x_+}$. Here, $x_\pm \in L_0 \cap L_1$. Then, this path is described by a map $u:I \times I \to M$ with some obvious Lagrangian boundary conditions. We can assign an index to $u$ in the following way.
	
	Let $E = u^*TM$ and choose a Lagrangian subbundle $F \subset E$ such that $F|_{(s,0)} = T_{u(s,0)} L_0$ and $F|_{(s,1)} = T_{u(s,1)} L_1$ for all $s$. After picking a trivialization, we can view the paths $u(s,0)$ and $u(s,1)$ being paths $I \to \mathcal{L}(n)$; call these $\gamma_0,\gamma_1$. Then, the index of $u$ can be defined as $I(u) := \mu(\gamma_0,\gamma_1)$. It is a result of the first two properties above that $I(u)$ depends only on homotopies of $u$ which keep the end points $\gamma_{x_\pm}$ fixed. So the choice of trivialization does not matter.
	
	Moreover, if $u,u'$ are two paths in $P(L_0,L_1)$ with the same endpoints, then $I(u) - I(u') = \chi(v)$. This $\chi \in H^1(P(L_0,L_1),\Z)$ is some class determined by the Maslov index for loops and $v$ is the loop obtained by concatenating $u$ and $u'$. We are interested in cases where this class $\chi = 0$; this happens, for example, when $c_1(M) = 0$ and $H^1(L_0) = H^1(L_1) = 0$. The latter condition ensures that the Maslov class, which obstructs the existence of gradings on Lagrangians, vanishes. For some details on this, see chapter 12 of \cite{Seidelbook}.
	
	When $\chi=0$, we can find numbers $i(\gamma_x) \in \frac{1}{2} \Z$ for each $x \in L_0 \cap L_1$. Then, if $u$ is a path between $\gamma_{x_+}$ and $\gamma_{x_-}$, then $I(u)= i(\gamma_{x_-})-i(\gamma_{x_+})$. Because of property (5), it can be arranged that
	
	\begin{equation} \label{index}
		i(\gamma_x) \equiv \frac{1}{2} \dim(T_xL_0 \cap T_xL_1) \pmod{1}
	\end{equation}
	
	\noindent Seidel calls numbers $i(\gamma_x)$ satisfying these properties \textit{coherent choices of indices.}
	
	We can also extend the above discussion to incorporate Hamiltonians. Choose two Hamiltonians $H_-,H_+$ and suppose that $\gamma_\pm$ are critical points of the actional functionals $\mathcal{A}_{H_\pm}$. Then if $u:I \to P(L_0,L_1)$ is a path between $\gamma_+$ and $\gamma_-$, there we can assign $u$ an index: $I_{H_-,H_+}(u) \in \frac{1}{2} \Z$. As before, if $\chi=0$, then there is a coherent choice of indices $i_H(\gamma) \in \frac{1}{2} \Z$ for any choice of Hamiltonian $H$ and critical point $\gamma$ of $\mathcal{A}_H$ so that
	\[I_{H_-,H_+}(u) = i_{H_-}(\gamma_{x_-})-i_{H_+}(\gamma_{x_+}). \]
	
	\noindent Furthermore, it can be arranged that
	\[i_H(\gamma) \equiv \frac{1}{2} \dim(D \phi^H_1 (T_{\gamma(0)}L_0) \cap T_{\gamma(1)} L_1) \pmod{1}.\]
	
	Also, if we begin with a coherent choice of indices $i(\gamma_x)$, we can choose $i_H(\gamma)$ such that $i(\gamma_x) = i_H(\gamma_x)$ when $H \equiv 0$.
	
	When $L_0,L_1$ are compact and the associated action functional is well-defined on $P(L_0,L_1)$ and also $\chi = 0$, then we have coherent choices of indices $i_H(\gamma)$. If $(H,J)$ are a regular pair, then we have a grading on the Floer homology groups $HF(L_0,L_1,H,J)$ which also induces gradings on local Floer homology.
	
	As was discussed at the beginning of Section 6, when $L_0 \cap L_1$ is a finite union of components $C_p$ with isolating neighborhoods $U_p$, then we have a filtration for the Floer chains and hence, a spectral sequence converging to $HF_*(L_0,L_1)$ with the $E^1$ page given by local Floer homology: $E^1_{p,q} = HF_{p+q}(L_0,L_1,U_p)$.
	
	If $L_0$ and $L_1$ have clean intersection, then property (4) above implies that for any coherent choice of indices the function $x \mapsto i(\gamma_x)$ is locally constant on $L_0 \cap L_1$. Let $i(C_p)$ be the value of this function on $C_p$ and $i'(C_p) = i(C_p)-\frac{1}{2} \dim C_p$. This is an integer because of Equation \ref{index}. Hence, $HF_*(L_0,L_1; U) \cong H_{*-i'(C)}(C, \Z_2)$ and the spectral sequence which converges to $HF_*(L_0,L_1)$ has $E^1$ page described completely by the topology of the $C_p$.
	
	\subsection{An Index Lemma}
	
	We wish to prove that $HF_*(L_0,L_1; U) \cong H_{*-\iota(C)}(C, \Z_2)$ in the case that $C$ is a minimally degenerate set and $\iota(C)$ is some type of Maslov index. From there, we will immediately have a spectral sequence as before. The results of Section 5 prove most of this result. We also know that there is a coherent choice of index $i(\gamma_x)$ for $x \in C$ that, because of property (4), is locally constant. What remains to be shown is that the index of $C$ and the index of a thickening $\Sigma_C$ are the same. 
	
	\begin{lemma}
		The index $i(C) = i(\Sigma_C)$. 
	\end{lemma}
	
	\begin{proof}
		Since minimally degenerate intersections are modeled locally on minimally degenerate functions, let $f: L \to \R$ be a minimally degenerate function and $C$ a minimally degenerate isolated critical locus with $S_C$ as the associated submanifold. Let $x \in S_C$. Then, the $H_x f$ is positive semi-definite along $T_x S_C$ but negative definite along some subspace $V$ transverse to $T_x S_C$ in $T_x L$. Perturbing $f$ gives a new function $\tilde{f}$ whose critical locus is a thickening $\Sigma_C$ of $C$. As a reminder, $\Sigma_C \subset S_C$ is a codimension 0 submanifold. The thickening is such that $H_x \tilde{f}$ is zero along $T_x S_C$ and negative definite along $V$. In particular this perturbation gets rid of all the positive eigenvalues of the Hessian of $f$ at $x$ but the dimension of the negative eigenspace stays the same. As a result, the index of $f$ and $\tilde{f}$ at $x$ are the same and equal $\dim V$. Here, index simply means the dimension of the negative eigenspace. So the perturbation does not change the ``Morse-type'' index.
		
		Next, we choose a regular $C^2$-small Hamiltonian $H$ and obtain a coherent choice of index $i_H(\gamma)$ for $\gamma$ being a critical point of $\mathcal{A}_H$. Because $H$ is $C^2$-small, Theorem \ref{Morse} says that for $x \in L_0 \cap L_1$, $\gamma_x$ is a critical point of $\mathcal{A}_H$. We may further conclude that $i(\gamma_x) = i_H(\gamma_x)$ as the Morse and Floer theoretic data coincide. 
		
		Therefore, the Maslov index will equal this Morse-type index up to a shift. Since the Morse-type index does not change, neither does the Maslov index as the Hamiltonian is $C^2$ small.
	\end{proof}
	
	With this lemma, we have the following result which is almost verbatim the statement by Seidel:
	
	\begin{theorem} \label{qmdspecseq}
		Suppose $L_0 \cap L_1$ decomposes into $\bigsqcup C_p$ where each $C_p$ is quasi-minimally degenerate. Let $\Sigma_p := \Sigma_{C_p}$ be the thickening for $C_p$. Then there is a spectral sequence which converges to $HF_*(L_0,L_1)$ and whose $E^1$-term is
		\[E^1_{pq} =
		\begin{cases}
			H_{p+q-\iota(\Sigma_p)}(C_p;\Z/2), & 1 \leq p \leq r; \\
			0, & \text{otherwise}.
		\end{cases}\]
	\end{theorem}
	
	\noindent This is a homology spectral sequence, i.e. the $k$-th differential $(k \geq 1)$ has degree $(-k,k-1)$.

\noindent \textbf{Remark:} We've presented the theorem in this way because that was the setting of Po\'zniak but as indicated in the proofs of Section \ref{adapt}, we can drop the compactness condition. Thus, if there are infinitely many QMD disjoint sets $C_p$ that form $L_0 \cap L_1$, so long as they are isolated, we can use the action filtration to build a finite $E^1$ page with action less than $A$ and form a directed system of finite $E^1$ pages. Then, since direct limits commute with homology, the limit $E^1$ page indeed is the $E^1$ page of a spectral sequence. Perhaps a more high-brow way to say this is that for any  additive functor $F$ between abelian categories induces a functor $F_*$ on the corresponding chain complex categories which preserves quasi-isomorphisms and commutes with homology. In this case, the the functor is $\text{colim}: \mathcal{C} \to \bf{AbGrps}$ where $\mathcal{C}$ is a filtered abelian category that we construct from our directed system of $E^1$ pages.

In applications where we introduce a Hamiltonian, say, in Hamiltonian Floer theory (we can switch back and forth between Hamiltonian and Lagrangian settings by Section ), for each cutoff $A$, there are only finitely many isolated families $C_p$. This is because we're working not on the manifold but rather the loop (or path) space and thus, if two loops have different action, we can find isolating neighborhoods.
	
	\section{Some Applications}
	
	In this section, we present some applications of this method. Several of these examples have already been studied but often with rather \textit{ad hoc} tools. Our purpose in presenting them is to demonstrate that our result can be used as a general and systematic tool. Before we do that, we'll first sketch why manifolds-with-corners are examples of QMD sets because these appear in the first two examples.
	
	Recall that a manifold-with-corners (or more succinctly, cornered manifolds) is a topological space $M = \mathring{M} \sqcup \partial M$ which decomposes into two pieces. The first piece is the interior $\mathring{M}$ which is an open manifold of, say, dimension $n$. The other piece, $\partial M$, also decomposes into pieces which are all manifolds themselves. The pieces have positive codimension and every point $p \in \partial M$ admits a neighborhood homeomorphic to $[0,1)^k \times (0,1)^{n-k}$ for some $0< k \leq n$.
	
	Now, working in this local model, let $\R^k$ have coordinates $x_1,...,x_k$ and $f(x_1,...,x_k) = \prod_{i=1}^k x_i$. Then $f^{-1}(0)$ is a union of hyperplanes and removing this union divides up $\R^k$ into $2^k$ connected components with one of them being distinguished as the \textbf{positive orthant} where all the $x_i \geq 0$; it is a cornered manifold. So then, the set $f^{-1}((\epsilon,\infty))$ intersected with the positive orthant gives us a manifold-with-boundary (more succinctly, a bordered manifold) which is a smoothing of the cornered manifold and the function gives us a deformation retracts onto the positive orthant. Moreover, note that $f$ does not have any critical points outside of $f^{-1}(0)$, the union of hyperplanes. 
	
	With some modifications to this scenario (which we do not write in detail), we can make a function $\tilde{f} \geq 0$ so that the positive orthant is $\tilde{f}^{-1}(0)$ is the set of minima and the thickening is diffeomorphic to the smoothing (by an isotopy). Essentially, it is the smoothing but we work in such a way that the thickening includes the positive orthant. This is all doable once we leave this model behind and work with more general manifolds because a cornered submanifold $C$ lives in an ambient manifold and locally, we may as well think of the ambient manifold as $\R^k \times \R^{n-k}$. We'll now turn to our examples.
	
	\begin{example}
		Consider $M = \C^*$ which is $\mathbb{CP}^1$ minus two points. Then $\C^* \times \C^*$ can be thought of as $\bb{P}^1 \times \bb{P}^1$ minus four projective lines which is, of course, an affine variety.
		
		Now, view $\C^* \cong \R_s \times S^1_t$ as a cylinder and pick real constants $a_1,b_1 \in \R$. We define a Hamiltonian $H_1(s,t) = H_1(s)$ which is $0$ on $[a_1,b_1]$ and $\frac{d}{ds}H_1(s) < 0$ when $s < a_1$ and $\frac{d}{ds}H_1(s) > 0$ when $s > b_1$. Then the critical set of $H_1$ is $A_1 = [a,b] \times S^1$, an annulus. We may do something similar on the other copy of $\C^*$ to produce a Hamiltonian $H_2$ with the critical set being an annulus, $A_2$.
		
		\begin{center}
			\tikzset{every picture/.style={line width=0.75pt}} 
			
			\tikzset{every picture/.style={line width=0.75pt}} 
			
			\begin{tikzpicture}[x=0.75pt,y=0.75pt,yscale=-1,xscale=1]
				
				\draw   (251.17,142.83) .. controls (251.17,142.83) and (251.17,142.83) .. (251.18,142.83) -- (434.33,142.83) .. controls (449.89,142.83) and (462.5,160.56) .. (462.5,182.42) .. controls (462.5,204.28) and (449.89,222) .. (434.33,222) -- (251.17,222) .. controls (235.61,222) and (223,204.28) .. (223,182.42) .. controls (223,160.56) and (235.61,142.83) .. (251.17,142.83) -- cycle ;
				\draw  [draw opacity=0] (251.17,142.83) .. controls (251.21,142.83) and (251.26,142.83) .. (251.3,142.83) .. controls (267.87,142.83) and (281.3,160.55) .. (281.3,182.42) .. controls (281.3,204.28) and (267.87,222) .. (251.3,222) .. controls (251.25,222) and (251.19,222) .. (251.14,222) -- (251.3,182.42) -- cycle ; \draw   (251.17,142.83) .. controls (251.21,142.83) and (251.26,142.83) .. (251.3,142.83) .. controls (267.87,142.83) and (281.3,160.55) .. (281.3,182.42) .. controls (281.3,204.28) and (267.87,222) .. (251.3,222) .. controls (251.25,222) and (251.19,222) .. (251.14,222) ;
				\draw  [draw opacity=0][dash pattern={on 0.84pt off 2.51pt}] (311.17,142.83) .. controls (311.21,142.83) and (311.26,142.83) .. (311.3,142.83) .. controls (327.87,142.83) and (341.3,160.55) .. (341.3,182.42) .. controls (341.3,204.28) and (327.87,222) .. (311.3,222) .. controls (311.25,222) and (311.19,222) .. (311.14,222) -- (311.3,182.42) -- cycle ; \draw  [dash pattern={on 0.84pt off 2.51pt}] (311.17,142.83) .. controls (311.21,142.83) and (311.26,142.83) .. (311.3,142.83) .. controls (327.87,142.83) and (341.3,160.55) .. (341.3,182.42) .. controls (341.3,204.28) and (327.87,222) .. (311.3,222) .. controls (311.25,222) and (311.19,222) .. (311.14,222) ;
				\draw  [draw opacity=0][dash pattern={on 0.84pt off 2.51pt}] (373.17,142.83) .. controls (373.21,142.83) and (373.26,142.83) .. (373.3,142.83) .. controls (389.87,142.83) and (403.3,160.55) .. (403.3,182.42) .. controls (403.3,204.28) and (389.87,222) .. (373.3,222) .. controls (373.25,222) and (373.19,222) .. (373.14,222) -- (373.3,182.42) -- cycle ; \draw  [dash pattern={on 0.84pt off 2.51pt}] (373.17,142.83) .. controls (373.21,142.83) and (373.26,142.83) .. (373.3,142.83) .. controls (389.87,142.83) and (403.3,160.55) .. (403.3,182.42) .. controls (403.3,204.28) and (389.87,222) .. (373.3,222) .. controls (373.25,222) and (373.19,222) .. (373.14,222) ;
				\draw    (251.3,35.83) .. controls (259.33,99.83) and (290.33,122.83) .. (317.33,122.83) ;
				\draw    (317.33,123) -- (379,123) ;
				\draw    (379,123) .. controls (407.33,121.83) and (425.33,95.83) .. (434.33,35.83) ;
				\draw   (370.17,142.83) .. controls (370.17,141.13) and (371.55,139.75) .. (373.25,139.75) .. controls (374.95,139.75) and (376.33,141.13) .. (376.33,142.83) .. controls (376.33,144.54) and (374.95,145.92) .. (373.25,145.92) .. controls (371.55,145.92) and (370.17,144.54) .. (370.17,142.83) -- cycle ;
				\draw  [draw opacity=0][fill={rgb, 255:red, 155; green, 155; blue, 155 }  ,fill opacity=0.25 ] (314.25,142.83) -- (373.14,142.83) -- (373.14,222) -- (314.25,222) -- cycle ;
				\draw  [draw opacity=0][fill={rgb, 255:red, 255; green, 255; blue, 255 }  ,fill opacity=1 ][dash pattern={on 0.84pt off 2.51pt}] (311.17,142.83) .. controls (311.21,142.83) and (311.26,142.83) .. (311.3,142.83) .. controls (327.87,142.83) and (341.3,160.55) .. (341.3,182.42) .. controls (341.3,204.28) and (327.87,222) .. (311.3,222) .. controls (311.25,222) and (311.19,222) .. (311.14,222) -- (311.3,182.42) -- cycle ; \draw  [dash pattern={on 0.84pt off 2.51pt}] (311.17,142.83) .. controls (311.21,142.83) and (311.26,142.83) .. (311.3,142.83) .. controls (327.87,142.83) and (341.3,160.55) .. (341.3,182.42) .. controls (341.3,204.28) and (327.87,222) .. (311.3,222) .. controls (311.25,222) and (311.19,222) .. (311.14,222) ;
				\draw   (310.17,142.83) .. controls (310.17,141.13) and (311.55,139.75) .. (313.25,139.75) .. controls (314.95,139.75) and (316.33,141.13) .. (316.33,142.83) .. controls (316.33,144.54) and (314.95,145.92) .. (313.25,145.92) .. controls (311.55,145.92) and (310.17,144.54) .. (310.17,142.83) -- cycle ;
				\draw  [draw opacity=0][fill={rgb, 255:red, 155; green, 155; blue, 155 }  ,fill opacity=0.25 ][dash pattern={on 0.84pt off 2.51pt}] (373.17,142.83) .. controls (373.21,142.83) and (373.26,142.83) .. (373.3,142.83) .. controls (389.87,142.83) and (403.3,160.55) .. (403.3,182.42) .. controls (403.3,204.28) and (389.87,222) .. (373.3,222) .. controls (373.25,222) and (373.19,222) .. (373.14,222) -- (373.3,182.42) -- cycle ; \draw  [dash pattern={on 0.84pt off 2.51pt}] (373.17,142.83) .. controls (373.21,142.83) and (373.26,142.83) .. (373.3,142.83) .. controls (389.87,142.83) and (403.3,160.55) .. (403.3,182.42) .. controls (403.3,204.28) and (389.87,222) .. (373.3,222) .. controls (373.25,222) and (373.19,222) .. (373.14,222) ;
				
				\draw (274,70.4) node [anchor=north west][inner sep=0.75pt]    {$H_{i}$};
				\draw (316,175.4) node [anchor=north west][inner sep=0.75pt]    {$a_{i}$};
				\draw (411,175.4) node [anchor=north west][inner sep=0.75pt]    {$b_{i}$};
				\draw (199,180.4) node [anchor=north west][inner sep=0.75pt]    {$...$};
				\draw (474,180.4) node [anchor=north west][inner sep=0.75pt]    {$...$};	
			\end{tikzpicture}
			\vspace{3mm}
			
			\caption{Hamiltonian $H_i$ on a cylinder; the shaded region is the critical set}
		\end{center}
		
		\noindent Then letting $H = H_1 + H_2$ on $M$, the critical set is $A_1 \times A_2$ which is a manifold-with-corners. Since the graph of the time-1 flow of $H$, when intersected with the diagonal $\Delta \subset M \times M$, is precisely $A_1 \times A_2$, we can move everything over to the Lagrangian Floer setting as briefly outlined in Section \ref{hamfloer}. Since the critical set is a manifold-with-corners, it is not Morse-Bott. But it is QMD in the Lagrangian sense. This is because away from the boundary and corners, the critical set is Morse-Bott and so we only need to focus on the codim 1 stratum. We can always find a Hamiltonian with which to perturb the critical set into a manifold-with-boundary (so we smooth out the corners) which puts us in the Lagrangian flattened degenerate situation. Lastly, being flattened degenerate automatically implies QMD.

		If we prefer, we can proceed with computing the Lagrangian Floer homology of the graph intersecting the diagonal and recover Hamiltonian Floer homology that way. But the present situation is simple enough here that we can just pick a Morse function $f$ and small constant $\epsilon > 0$. Perturbing $H$ by $\epsilon f$ breaks the critical set into isolated points. Then, 
		\[HF_*(A_1 \times A_2) \cong HM_*(A_1 \times A_2) \cong H^{sing}_*(A_1) \otimes H^{sing}_*(A_2). \]
		
		\noindent Here, $HM_*$ is Morse homology and of course, we have a K\"unneth formula.
		
	\end{example}
	
	\begin{example}
		Let $X$ be a smooth complex affine variety; Hironaka's theorem gives a compactification of $X$ in the sense that $X \cong M \setminus D$ where $M$ is a smooth projective variety and $D$ is a simple normal crossings divisor which supports an ample line bundle. In a paper by Ganatra and Pomerleano \cite{GanatraPomerleano}, their main result produces a spectral sequence which converges to the symplectic cohomology $SH^*(X)$ and its $E_1$ page is ring isomorphic to the logarithmic cohomology of $(M,D)$:
		\[H^*_{log}(M,D) \cong \bigoplus_{p,q} E^{p,q}_1\]
		
		\noindent The $E^{p,q}_1$ are formed from local Floer homology groups. In order to obtain the isomorphism on the group level, they conduct a study of families of Hamiltonian orbits which are manifolds with corners via a variant of Morse-Bott analysis. An alternative approach to their study is to choose a neighborhood on which to perturb a manifold with corners, thereby smoothing out the corners to become a manifold with boundary (and thus, placing us in the QMD setting like the previous example). One then computes the local Hamiltonian Floer groups by translating to the Lagrangian setting. Ganatra and Pomerleano were aware of this alternative approach which they allude to in Remark 4.17 on p. 72 (arxiv version) in their paper. Since local Floer homology is invariant under such perturbations, we may freely perturb in this manner. \\
		
		\noindent \textbf{Remark:} In order, to produce the multiplicative structure, Ganatra and Pomerleano produce a novel log PSS map which they developed in a prior paper \cite{GP}.
	\end{example}
	
	\begin{example}
		This next example was studied by Pascaleff in his thesis \cite{Pascaleff1} and we will spend considerably more time on it. He computed the wrapped Floer homology of certain Lagrangians, including the ring structure that comes from counting Floer triangles. Here, we will give a weaker result as a technical demonstration of the principles from above. The purpose of choosing this example is to compare these methods to known results and to also fill in some details of Pascaleff's work. For some excellent pictures, consult \cite{Pascaleff1}, \cite{Pascaleff2}.
		
		Consider a line $L$ and conic $C$ in $\mathbb{CP}^2$ which intersect transversally. Letting $D = L \cup C$, this is an anticanonical divisor of $\mathbb{CP}^2$ and has the properties needed to view the pair $(\mathbb{CP}^2,D)$ as a log Calab-Yau. For instance, $D$ is a normal crossings divisor. Further details of the definition are found in Pascaleff's thesis.
		
		Next, we opt to blowup the two points of intersection since any blowup along $D$ will not affect the $\mathbb{CP}^2 \setminus D$. What we obtain then is a new divisor $\widetilde{D}$ inside of the twice blowup, call it $X$, which is the union of the proper transforms $\widetilde{L}$ and $\widetilde{C}$ as well as two exceptional divisors, $E$ and $F$. Below is the toric picture where we take $\mathbb{CP}^2$, represented by its moment polytope $\Delta$. The preimate $\mu^{-1}(\partial \Delta)$ of the boundary under the moment map is a union of three lines. We can smooth one of the corners, the bottom left one at the cost of introducing a nodal fiber, marked with an x.
		
		\begin{center}
			\tikzset{every picture/.style={line width=0.75pt}} 
			
			\begin{tikzpicture}[x=0.75pt,y=0.75pt,yscale=-1,xscale=1]
				
				\draw    (80,34.33) -- (205.67,160) ;
				\draw    (80,160) -- (205.67,160) ;
				\draw    (80,34.33) -- (80,160) ;
				\draw  [dash pattern={on 0.84pt off 2.51pt}]  (80,160) -- (110,130) ;
				\draw    (110,120) -- (120,130) ;
				\draw    (120,120) -- (110,130) ;
				\draw [color={rgb, 255:red, 0; green, 0; blue, 255 }  ,draw opacity=1 ]   (86.01,103.78) -- (86.01,123.78) ;
				\draw [color={rgb, 255:red, 0; green, 0; blue, 255 }  ,draw opacity=1 ]   (121.67,84.83) -- (151.67,114.83) ;
				\draw [color={rgb, 255:red, 0; green, 0; blue, 255 }  ,draw opacity=1 ]   (86.01,103.78) .. controls (84.67,48.83) and (90.67,49.83) .. (121.67,84.83) ;
				\draw [color={rgb, 255:red, 0; green, 0; blue, 255 }  ,draw opacity=1 ]   (126.01,153.78) .. controls (192.67,152.83) and (197.67,162.83) .. (151.67,114.83) ;
				\draw [color={rgb, 255:red, 0; green, 0; blue, 255 }  ,draw opacity=1 ]   (86.01,123.78) .. controls (85.68,154.61) and (93.68,154.61) .. (126.01,153.78) ;
				\draw    (258,50) -- (344,140) ;
				\draw    (238,160) -- (344,160) ;
				\draw    (238,50) -- (238,160) ;
				\draw  [dash pattern={on 0.84pt off 2.51pt}]  (238,160) -- (268,130) ;
				\draw    (268,120) -- (278,130) ;
				\draw    (278,120) -- (268,130) ;
				\draw [color={rgb, 255:red, 0; green, 0; blue, 255 }  ,draw opacity=1 ]   (244.01,103.78) -- (244.01,123.78) ;
				\draw [color={rgb, 255:red, 0; green, 0; blue, 255 }  ,draw opacity=1 ]   (279.67,84.83) -- (309.67,114.83) ;
				\draw [color={rgb, 255:red, 0; green, 0; blue, 255 }  ,draw opacity=1 ]   (244.01,103.78) .. controls (242.67,48.83) and (248.67,49.83) .. (279.67,84.83) ;
				\draw [color={rgb, 255:red, 0; green, 0; blue, 255 }  ,draw opacity=1 ]   (284.01,153.78) .. controls (350.67,152.83) and (355.67,162.83) .. (309.67,114.83) ;
				\draw [color={rgb, 255:red, 0; green, 0; blue, 255 }  ,draw opacity=1 ]   (244.01,123.78) .. controls (243.68,154.61) and (251.68,154.61) .. (284.01,153.78) ;
				\draw    (344,140) -- (344,160) ;
				\draw    (258,50) -- (238,50) ;
			
			\end{tikzpicture}
		
		\caption{Toric Picture with $\mathbb{CP}^2$ and $Bl_2 \mathbb{CP}^2$}
		\end{center}
		
		The reason for this blowup is because it gives us control over neighborhoods of the intersections of these four curves. In particular, we may choose holomorphic coordinates $(z,w)$ such that the a neighborhood of an intersection appears as $\C^2$ with one divisor locally appearing as $\{z = 0\}$ and the other as $\{w = 0\}$. The blowup parameters allow us to obtain a symplectic form which sees the two divisors as symplectically orthogonal. This is a consequence of the $U(2)$ invariance of the blowup form. In fact, we only need $U(1) \times U(1)$ invariance. The argument here is essentially what is outlined in \cite{Seidel1}, Theorem 4.5.
		
		By a relative Moser argument, we are able to extend the symplectic form to a neighborhood $U$ of $\widetilde{D}$ without disrupting the symplectic orthogonality of the curves. Thus, we have another log Calabi-Yau $(X,\widetilde{D})$ and we shall study $X \setminus \widetilde{D}$ or more precisely, $X \setminus U$. In order to do this, let us give a more refined view of $U$. The main issue is to consider neighborhoods of the nodes as there are concerns about smoothness. We take polar coordinates for $\C^2$, $(r_1,\theta_1,r_2, \theta_2)$ and consider the real hypersurface $\{r_1 r_2 = \delta \}$ for some small $\delta > 0$. Then in this locale, $U$ may be thought of as $\{r_1 r_2 < \delta\}$. When we extend the hypersurface, it gives a 3-manifold $M$ which is, in fact, a $T^2$ bundle over $S^1$. This is clear from the toric picture where the blue curve represents the $S^1$ over which $M$ is $T^2$-fibered. It is then also clear that it doesn't matter whether we use $\mathbb{CP}^2$ or its blowup at two points. This was presented in Section 7 of Pascaleff's thesis.
		
		$M$ can be classified by an element of $SL(2,\Z)$ which gives the monodromy. The normal bundles of the curves in $X$ are $\mathcal{O}(-1) \to \mathbb{CP}^1$ for the exceptional divisors and the proper transform $\widetilde{L}$. For $\widetilde{C}$, it is $\mathcal{O}(2) \to \mathbb{CP}^1$. To construct $M$, we need to ``plumb'' the circle bundles of these normal bundles together. The map to plumb these bundles at the nodes corresponds to
		\[J = \begin{pmatrix}
			0 & -1 \\
			1 & 0
		\end{pmatrix}\]
		
		\noindent since we are basically interchanging circles fibers in an orientation-preserving way. On the other hand, by choosing meromorphic sections with single poles for $\mathcal{O}(-1)$ and a holomorphic section with two zeros for $\mathcal{O}(2)$, we obtain some contributions to the monodromy map. Let 
		\[T = \begin{pmatrix}
			1 & 1 \\
			0 & 1
		\end{pmatrix}.\]
		
		\noindent Then, the sections above correspond to $T^{-1}$ and $T^2$, respectively. Thus, the monodromy map (in this basis) is given by multiplication of a sequence of these matrices: 
		\[\mu = JT^{-1}JT^{-1}JT^{-1}J T^2 =
		\begin{pmatrix}
			2 & 1 \\
			-1 & 0
		\end{pmatrix}.\]
		
		The vector field $V = r_1 \partial_{\theta_1} + r_2 \partial_{\theta_2}$ gives a characteristic foliation of $M$ and is tangent to the fibers. Pascaleff wrote down a contact form $\alpha$ which realizes $V$ as its Reeb vector field. We may also write down a Liouville vector field in order to produce a Liouville domain which appears as the affine variety $X \setminus U$ with contact boundary given by $M$.
		
		In the sequel, our goal is to study the wrapped Floer homology of the Lagrangian defined by removing the neighborhood of the divisor from the real part of $X$. For the definition of wrapped Floer homology, one may consult \cite{McLeanGrowth}, \cite{Pascaleff1}. One can describe the real part as the fixed point set of an antisymplectic involution which coincides with complex conjugation away from the blowup points and so, topologically, it will be $\mathbb{RP}^2 \setminus D$ where $D$ is the real part of the conic plus line; of course, this is the same thing as removing a conic from $\R^2$. Call this Lagrangian $\Lambda$.
		
		$\Lambda$ is cylindrical at infinity which means it is, at infinity, the product of a Legendrian submanifold of the contact boundary, product with $\R$. This parameter $\R$ can be thought of as changing $\delta$. The submanifold is a 1-manifold and is disconnected. Indeed, if we look at the real picture of $\Lambda$, the real part of the divisor separates $\Lambda$ into two or three components, depending on the real conic. So $\Lambda$ has two or three cylindrical ends. Each boundary component is a Legendrian knot and may be viewed as a section of the $T^2$ bundle over $S^1$.
		
		Near the nodes, $V = r_1 \partial_{\theta_1} + r_2 \partial_{\theta_2}$. $\Lambda$ intersects each torus fiber at four points since the real part requires $\theta_1,\theta_2 \in \{0,\pi\}$. If we view the torus as $\R^2/(2 \pi\Z)^2 = [0,1]^2/\sim$, then the points are $(0,0),(\frac{1}{2},0), (0,\frac{1}{2}), (\frac{1}{2},\frac{1}{2})$. We have Reeb orbits whenever $r_2/r_1 \in \Q$. However, we also have Reeb chords between the distinct points depending on $r_1$ and $r_2$. We'll let $q=r_1, p = r_2$. It can be easily checked that
		
		\begin{itemize}
			\item $(0,0)$ connects to $(0, \frac{1}{2})$ if and only if $q$ is even.
			
			\item Whether $(0,0)$ connects to $(\frac{1}{2},0)$ is symmetric to the above. We have solutions exactly when $p$ is even.
			
			\item $(0,0)$ connects to $(\frac{1}{2},\frac{1}{2})$ exactly when both $p$ and $q$ are odd.
			
			\item Whether $(\frac{1}{2},0)$ connects to $(0,\frac{1}{2})$ is equivalent to the previous; think of $(\frac{1}{2},0)$ as $(\frac{1}{2},1)$, then translate the plane down by a half.
			
			\item Whether $(\frac{1}{2},0)$ connects to $(\frac{1}{2},\frac{1}{2})$ or whether $(0,\frac{1}{2})$ connects to $(\frac{1}{2},\frac{1}{2})$ is the same as some of the above situations; just translate the plane.
		\end{itemize}
			
		\noindent We also note that the lengths of the orbits/chords depends both on the numerator and denominator in $r_2/r_1$; this means that we're able to isolated the families via length of orbit. Here is a picture where $p,q$ are both odd.
		
		\begin{center}
			
			\tikzset{every picture/.style={line width=0.75pt}} 
			
			\begin{tikzpicture}[x=0.75pt,y=0.75pt,yscale=-1,xscale=1]
				
				\draw   (170.83,30.19) -- (320.83,30.19) -- (320.83,181.19) -- (170.83,181.19) -- cycle ;
				\draw    (259.83,30.19) -- (170.83,181.19) ;
				\draw    (320.83,79.19) -- (259.83,181.19) ;
				\draw    (170.83,81) -- (200.83,30.19) ;
				\draw    (290.33,30.19) -- (201.33,181.19) ;
				\draw    (290.83,181) -- (320.83,130.19) ;
				\draw    (231.83,30.19) -- (170.83,131.19) ;
				\draw    (320.83,30.19) -- (229.83,181.19) ;
				\draw   (168.24,181.19) .. controls (168.24,179.76) and (169.4,178.6) .. (170.83,178.6) .. controls (172.27,178.6) and (173.43,179.76) .. (173.43,181.19) .. controls (173.43,182.63) and (172.27,183.79) .. (170.83,183.79) .. controls (169.4,183.79) and (168.24,182.63) .. (168.24,181.19) -- cycle ;
				\draw   (243.24,181.19) .. controls (243.24,179.76) and (244.4,178.6) .. (245.83,178.6) .. controls (247.27,178.6) and (248.43,179.76) .. (248.43,181.19) .. controls (248.43,182.63) and (247.27,183.79) .. (245.83,183.79) .. controls (244.4,183.79) and (243.24,182.63) .. (243.24,181.19) -- cycle ;
				\draw   (168.24,105.19) .. controls (168.24,103.76) and (169.4,102.6) .. (170.83,102.6) .. controls (172.27,102.6) and (173.43,103.76) .. (173.43,105.19) .. controls (173.43,106.63) and (172.27,107.79) .. (170.83,107.79) .. controls (169.4,107.79) and (168.24,106.63) .. (168.24,105.19) -- cycle ;
				\draw   (242.64,105.69) .. controls (242.64,104.26) and (243.8,103.1) .. (245.24,103.1) .. controls (246.67,103.1) and (247.83,104.26) .. (247.83,105.69) .. controls (247.83,107.13) and (246.67,108.29) .. (245.24,108.29) .. controls (243.8,108.29) and (242.64,107.13) .. (242.64,105.69) -- cycle ;
			\end{tikzpicture}
			
			\caption{Example of Reeb chords with slope $p/q=5/3$}
		\end{center}	
		
		If we pick an admissible Hamiltonian $H$, this shows that there are plenty of generators for the wrapped Floer chain complex $CW_*(\Lambda,\Lambda,H)$. Away from the nodes, the situation is tamer and the total space $M$ admits a Boothby-Wang structure (we can ensure the symplectic form is integral). The circle action gives Morse-Bott submanifolds similar to the first example above of $\C^* \times \C^*$ (and hence, we're in the QMD setting).
		
		Pascaleff gives reasons for why the generators are all concentrated in degree zero and thus, why the differential of $CW_*(\Lambda,\Lambda, H)$ is trivial; he then computes that $CW_0(\Lambda,\Lambda,H) \cong HW_*(\Lambda,\Lambda) \cong \mathbb{K}[x,y][(xy- 1)^{-1}]$. We can supply some evidence for this from a local Floer theoretic perspective. If we restrict our attention to low energy Floer strips, they must connect Reeb chords of some ``type'' to Reeb chords of the same ``type'' because of the boundary conditions. To be more precise, consider an interval $I \subset S^1$ and $T^2 \times I$. Then since $\Lambda$ intersects $T^2$ at four points, one can imagine four line segments in $T^2 \times I$ which represent the intersection with $\Lambda$. The boundary conditions imposed on Floer strips makes it so that these low energy strips must connect a Reeb chord to a translation of the Reeb chord along $I$. Thus, regardless of what the degree of the Reeb chord is, the Floer strip is not connecting Reeb chords of differing index. As such, in the local Floer complex where we take only low energy strips, there is no differential in the low-energy regime. This fact shows that we may obtain the $E^1$ page of the spectral sequence from the local Floer data. Moreover, the rank of the underlying vector space of the $E^1$ page is countably infinite which corroborates Pascaleffs calculation and so the differentials in the spectral sequence must have large kernels and not too large of images. The countable infinity is not a problem for us; this was addressed in the remark following Theorem \ref{qmdspecseq}.
		
		To summarize, Pascaleff was able to compute the wrapped Floer homology without this spectral sequence and gave a description of the triangle product to determine the ring structure. To do this, he relies on the example being a log Calabi-Yau pair which allowed him to use some mirror symmetry techniques. Here, we've given a weaker understanding of the example by using our spectral sequence but without reliance on the log Calabi-Yau condition nor mirror symmetry. Notably, the spectral sequence can be applied to non-compact Lagrangians.
	\end{example}
	
	\begin{example}
		The spectral sequence we construct from the local Floer data applies to examples beyond log Calabi-Yau. For instance, if we have any complex algebraic surface, so long as we have an antisymplectic involution that has fixed points, the fixed point set is a Lagrangian. It is convenient to take the one which fixes the real locus but is certainly not the only option.
		
		Now, let's take four generic lines in $\mathbb{CP}^2$ fixed by an antisymplectic involution. Then the union of them gives a divisor that is not anticanonical. Any given line will intersect all the other three lines, giving a total of six intersections. Blowing up at those points, we'll obtain six exceptional lines for a total of 12 nodes. We apply our study of the nodes from the previous example to these 12. Away from the nodes, the topology differs from the previous example. Before, the Morse-Bott manifolds-with-boundary were all annuli. This time, the proper transforms of the original lines give thrice-punctured spheres while the exceptional divisors continue to contribute annuli. So the local Floer data which feeds into the spectral sequence is different from before but still tractable.
		
		There are certainly many other line arrangements which produce a multitude of affine varieties and Lagrangians to which we may also apply these techniques.
	\end{example}

	\section{Concluding Remarks}
	
	Hopefully, the above examples are not exhaustive. There may be situations in which these techniques would be helpful for computing triangle products or other $A^\infty$ products on local Lagrangian Floer homology.
	
	Minimal degeneracy may also appear naturally in low-dimensional topology. One situation of particular interest was communicated to the author by Kenji Fukaya who suggested some possible relationships to Atiyah-Floer conjectures. Namely, the instanton homology of $M$, an $S^1$ bundle over a Riemann surface, has a Chern-Simons functional has behavior similar to those appearing in Kirwan’s study of moment maps. If we take a Heegaard decomposition of $M$, the moduli of flat connections on $M$ is a certain Lagrangian intersection and the local properties of the Chern-Simons functional can be related to the mildness of the intersection. It is not Morse-Bott in general but may be QMD.
		
	Another direction could be that of contact homology. Bourgeois studied Morse-Bott techniques in his PhD thesis \cite{Bourgeois}. Perhaps one could find minimally degenerate or QMD contact forms and do computations once these notion have been properly defined. 
	
	Lastly, there may also be interesting questions about minimally degeneracy itself to explore. Holm and Karshon show in \cite{HolmKarshon} that Kirwan's definition of minimal degeneracy is local. That is, if $f: M \to \R$ is a smooth function and is minimally degenerate near each critical point, then $f$ is minimally degenerate. It may be worth exploring whether the same can be said of quasi-minimally functions in both the ``Morse'' and Lagrangian setting. It seems plausible for there to be obstructions for a Lagrangian intersection to be locally QMD at each point but not globally QMD. Such obstructions are likely to be completely topological.

\bibliographystyle{alpha}
\bibliography{ref}	
	
\end{document}